\documentclass[11pt,oneside,reqno]{amsart}
\usepackage{amsmath,amssymb, geometry, color,tikz-cd}
\usepackage[mathlines,pagewise]{lineno}
\usepackage{mathtools}
\usepackage[alphabetic]{amsrefs}

\usepackage[colorlinks=true, linkcolor=black, citecolor=black, urlcolor=blue]{hyperref}

\numberwithin{equation}{section}

\usepackage{bbm}
\usepackage{enumitem}
\usepackage{upgreek}

\newtheorem{proposition}{Proposition}[section]
\newtheorem{theorem}[proposition]{Theorem}
\newtheorem{lemma}[proposition]{Lemma}
\newtheorem{corollary}[proposition]{Corollary}

\theoremstyle{definition}
\newtheorem{remark}[proposition]{Remark}

\newtheorem{question}[proposition]{Question}

\title{Variation of cones of divisors in a family of varieties -- Fano type case}

\author{Sung Rak Choi}
\address[Sung Rak Choi]{Department of Mathematics, Yonsei University, 50 Yonsei-ro, Seodaemun-gu, Seoul 03722, Republic of Korea}
\email{sungrakc@yonsei.ac.kr}

\author{Zhan Li}
\address[Zhan Li]{Department of Mathematics, Southern University of Science and Technology, 1088 Xueyuan Road, Shenzhen 518055, China} \email{lizhan@sustech.edu.cn, lizhan.math@gmail.com}

\author{Chuyu Zhou}
\address[Chuyu Zhou]{School of Mathematical Sciences, Xiamen University, Siming South Road 422, Xiamen, Fujian 361005, China}
\email{chuyuzhou@xmu.edu.cn, chuyuzhou1@gmail.com}

\date{} % delete this line to display the current date

\subjclass[2020]{14J45, 14E30}
\keywords{Fano type variety, Nef cone, Movable cone, Mori chamber decomposition.}

\newcommand{\Spec}{{\rm {Spec}}}

\newcommand{\Exc}{{\rm {Exc}}}

\newcommand{\Pic}{{\rm {Pic}}}

\newcommand{\Supp}{{\rm {Supp}}}

\newcommand{\Nef}{{\rm {Nef}}}

\newcommand{\Eff}{{\rm {Eff}}}
\newcommand{\Mov}{{\rm {Mov}}}

\newcommand{\bc}{{\rm {bc}}}
\newcommand{\NE}{{\rm {NE}}}

\newcommand{\bEff}{\overline{\rm {Eff}}}
\newcommand{\bMov}{\overline{\rm {Mov}}}
\newcommand{\dto}{\dashrightarrow}

\newcommand{\bC}{\mathbb{C}}
\newcommand{\bD}{\mathbb{D}}

\newcommand{\bN}{\mathbb{N}}

\newcommand{\bP}{\mathbb{P}}
\newcommand{\bQ}{\mathbb{Q}}
\newcommand{\bR}{\mathbb{R}}

\newcommand{\bZ}{\mathbb{Z}}

\newcommand{\mB}{\mathcal{B}}

\newcommand{\mD}{\mathcal{D}}

\newcommand{\mG}{\mathcal{G}}
\newcommand{\mH}{\mathcal{H}}

\newcommand{\mN}{\mathcal{N}}
\newcommand{\mO}{\mathcal{O}}
\newcommand{\mP}{\mathcal{P}}

\newcommand{\mX}{\mathcal{X}}
\newcommand{\mY}{\mathcal{Y}}
\newcommand{\mZ}{\mathcal{Z}}

\begin{document}

\date{\today}

\begin{abstract}
We investigate the relationship between the Fano type property on fibers over a Zariski dense subset and the global Fano type property. We establish the invariance of N\'eron-Severi spaces, nef cones, effective cones, movable cones, and Mori chamber decompositions for a family of Fano type varieties after a dominant generically finite base change. Additionally, we show the uniform behavior of the minimal model program for this family. These results are applied to the boundedness problem of Fano type varieties.
\end{abstract}

\maketitle

\setcounter{tocdepth}{1}

\tableofcontents

\section{Introduction}\label{sec: intro}

We work over the field of complex numbers.

For a projective variety $X$, various cones of divisors can be associated within the N\'eron-Severi space $N^1(X)$. Notable among them are the effective cone $\Eff(X)$, the nef cone $\Nef(X)$, and the movable cone $\Mov(X)$. Moreover, the movable cone admits a finer decomposition known as the Mori chamber decomposition. These cones as well as their structures ``linearize" the geometric properties of the variety. 

Suppose that $X \to T$ is a family of varieties. It is natural to associate the aforementioned cones with each fiber and study how the deformation of the varieties affects the structure of the various cones. However, these cones may change drastically. For instance, there may not exist an open subset $U \subset T$ where the dimension of the N\'eron-Severi space $\dim N^1(X_t)$ remains constant for all $t \in U$.  

One motivation for this note is to apply the cone structure to study the moduli problem of Fano type varieties. Recall that a projective variety $X$ is of Fano type if there exists a divisor $\Delta$ such that $(X, \Delta)$ has klt singularities and $-(K_X+\Delta)$ is ample. Therefore, instead of considering a family of Fano type varieties, one should only assume the existence of a Zariski dense subset over which the fibers are of Fano type. In general, such a Zariski dense subset might be very sparse. Thus, it is natural to ask:

\begin{question}\label{que: FT fiber to FT}
Let $X \to T$ be a family of varieties. Suppose that the fibers are of Fano type over a Zariski dense subset of $T$. After shrinking $T$, is $X$ of Fano type over $T$?
\end{question}

This natural question seems to be rather subtle, and we can only provide a partial answer.

\begin{theorem}\label{thm: FT is constructible}
Let $f: X \to T$ be a projective, surjective morphism between varieties such that $f_*\mO_X=\mO_T$. Suppose that there exists a Zariski dense subset $S\subset T$ such that for any $s\in S$, the fiber $X_s$ is a Fano type variety. Then, after shrinking $T$, there exists a divisor $D$ on $X$ such that $(X, D)$ has lc singularities and
\[
K_X+D \sim_\bQ 0/T.
\] Furthermore, if for any $s\in S$, the fiber $X_s$ falls into one of the following cases:
\begin{enumerate}[label=(\roman*)]
\item $X_s$ is a klt weak Fano variety for any $s\in S$,
\item $X_s$ is of Fano type for any $s\in S$ and $S$ consists of very general points,
\item $X_s$ is of Fano type for any $s\in S$ and $\{X_s \mid s\in S\}$ admits bounded klt complements,
\end{enumerate}
then, after shrinking $T$, we have that $X$ is of Fano type over $T$.
\end{theorem}

In the above theorem, a variety $X$ is called klt weak Fano if $X$ has klt singularities and the anticanonical divisor $-K_X$ is nef and big; a set $S$ is said to consist of very general points if $S$ is the complement of the union of countably many proper Zariski subvarieties. For other notions and detailed explanations, see Section \ref{sec: pre}.

By the invariance of plurigenera, varieties of general type deform to varieties of general type (see \cite{HMX18}). The above theorem can be viewed as a weak analogy to this result for Fano type varieties. The proof of this natural theorem, however, requires deep results on complements as developed in \cite{Birkar19}.  With this result in mind, to study the deformation of cones of Fano type varieties in the above cases, we can always assume that each fiber is of Fano type after shrinking the base. It turns out that various cones exhibit nice behaviors in this family. As a first step, we show that, after a dominant generically finite base change and a suitable shrinking of the base, the N\'eron-Severi space of each fiber coincides with the restriction of the N\'eron-Severi space of the total space.

\begin{theorem}\label{thm: deform pic}
Let $f: X \to T$ be a projective fibration. Suppose that $S\subset T$ is a Zariski dense subset such that for any $s\in S$, the fiber $X_s$ has rational singularities and
\[
H^1(X_s, \mO_{X_s})=H^2(X_s, \mO_{X_s})=0.
\] 
\begin{enumerate}
    \item If the natural restriction map $N^1(X/T) \to N^1(X_t)$ is surjective for very general $t\in T$, then there exists a non-empty open subset $T_0\subset T$, such that ${\mP}ic_{X_{T_0}/T_0} \otimes \bR$ is a constant sheaf in the Zariski topology.
    \item Up to a dominant generically finite base change of $T$, there exists a non-empty open subset $T_0\subset T$, such that ${\mP}ic_{X_{T_0}/T_0} \otimes \bR$ is a constant sheaf in the Zariski topology. 
\end{enumerate} 
Moreover, in both of the above cases, for any open subset $U\subset T_0$ and any $t\in U$, the natural restriction maps
\[
N^1(X_{T_0}/T_0) \to N^1(X_U/U) \to N^1(X_t)
\] are isomorphisms.
\end{theorem}

In the above theorem, ``up to a dominant generically finite base change of $T$" means that we take a dominant generically finite morphism $T' \to T$ and replace $T$ with $T'$. The non-empty open subset $T_0$ is then taken inside this new base $T'$. Hence, after possibly further shrinking $T_0$, we are indeed performing an \'etale base change of an open subset of the original base. 

The proof of the above result follows from an argument akin to \cite{dFH11}*{\S 6}, where a similar result is established in a slightly different setting. Due to the monodromy phenomenon, even in a locally trivial family, there are no natural identifications between the N\'eron-Severi spaces of different fibers. Theorem \ref{thm: deform pic} enables comparisons between these spaces and their associated cones. 

The following fiber-wise small $\bQ$-factorial modification illustrates the necessity of considering families whose fibers have rational singularities and satisfy $h^i(X_s, \mO_{X_s}) = 0$ for $i = 1, 2$, rather than restricting only to varieties of Fano type. Indeed, these properties are preserved under resolutions, which may result in varieties that are no longer of Fano type. Such fiber-wise small $\bQ$-factorial modifications are essential in establishing the boundedness of moduli problems, although they have been used in the literature without proper justification.

\begin{theorem}\label{thm: fiberwise q-fact}
Let $f: X \to T$ be a projective fibration. Suppose that $S\subset T$ is a Zariski dense subset such that for any $s\in S$, the fiber $X_s$ has rational singularities and
\[
H^1(X_s, \mO_{X_s})=H^2(X_s, \mO_{X_s})=0.
\] 
Assume that there exists a divisor $\Delta$ such that $(X, \Delta)$ has klt singularities. Then, up to a dominant generically finite base change of $T$, there exist a birational morphism $Y \to X$ and a non-empty open subset $T_0 \subset T$ such that $Y \to X$ and each fiber $Y_t \to X_t$ for $t \in T_0$ are small $\bQ$-factorial modifications.

Moreover, for any open subset $U\subset T_0$ and any $t\in U$, the natural restriction maps
\[
\begin{split}
    &N^1(Y_{T_0}/T_0) \to N^1(Y_U/U) \to N^1(Y_t)\\
    &N^1(X_{T_0}/T_0) \to N^1(X_U/U) \to N^1(X_t)
\end{split}
\] are isomorphisms.
\end{theorem}

The following result establishes the generic deformation invariance of various cones up to a base change.  

\begin{theorem}\label{mainthm: decomposition}
Let $f: X \to T$ be a projective fibration. Suppose that there exists a Zariski dense subset $S\subset T$ such that for any $s\in S$, the fiber $X_s$ satisfies one of the cases in Theorem \ref{thm: FT is constructible}.
\begin{enumerate}
    \item If the natural restriction map $N^1(X/T) \to N^1(X_t)$ is surjective for very general $t\in T$, then there exists a non-empty open subset $T_0\subset T$, such that the N\'eron-Severi space $N^1(X_t)$ as well as the cones $\Eff(X_t)$, $\Nef(X_t)$ and $\Mov(X_t)$ are deformation invariant on $T_0$. Moreover, the Mori chamber decomposition on $\Mov(X_t)$ is also deformation invariant on $T_0$.
    \item Up to a dominant generically finite base change of $T$, there exists a non-empty open subset $T_0\subset T$, such that  the N\'eron-Severi space $N^1(X_t)$ as well as the cones $\Eff(X_t)$, $\Nef(X_t)$ and $\Mov(X_t)$ are deformation invariant on $T_0$. Moreover, the Mori chamber decomposition on $\Mov(X_t)$ is also deformation invariant on $T_0$.
\end{enumerate} 
\end{theorem}

In the above theorem, we will show that the natural restriction maps $N^1(X_{T_0}/T_0) \to N^1(X_t)$ is an isomorphism for any $t \in T_0$. Therefore the deformation invariance on $T_0$ means that the cones $\Eff(X_t)$, $\Nef(X_t)$ and $\Mov(X_t)$ are all identified with the corresponding cones in $N^1(X_{T_0}/T_0)$. Moreover, the Mori chamber decomposition on $\Mov(X_t)$ is identified with the Mori chamber decomposition on $\Mov(X_{T_0}/T_0)$. For the precise definition of the Mori chamber decomposition, see \eqref{eq: Mori chamber decomp1}. 

For a family of varieties, one can shrink the base and apply  Noetherian induction to obtain finitely many families of varieties such that the desired properties hold for each family. Moreover, this operation is necessary if one wants to establish the deformation invariance of nef cones even for a family of $\bQ$-factorial terminal Fano varieties (see \cite{Totaro12}). On the other hand, the nef cones are indeed locally constant (i.e., deformation invariant) for a smooth family of Fano manifolds (see \cites{Wis91, Wis09}). When $X \to T$ is a family of $\mathbb{Q}$-factorial terminal Fano varieties over a curve, \cite{dFH11} showed that the movable cones are locally constant. Under the same setting, the local constancy of the Mori chamber decompositions is also established for special families of Fano varieties, notably when $\dim(X/T) \leq 3$ (see counterexamples in \cite{Totaro12} for cases where $\dim(X/T) > 3$). Moreover, under various restrictions, the invariance of effective cones and movable cones is established in \cite{HX15}. \cite{Sho20}*{\S 4} established similar results when $X$ is of Fano type over $T$. In fact, \cite{Sho20}*{\S 4} showed that the sets of integral divisors associated with various cones remain invariant. The case of effective cones is also addressed in a recent preprint \cite{CHHX25}. These results are indispensable in the theory of complements.

Next, we explore the uniform behavior of the minimal model program (MMP) for a family of varieties. In the case where the fibers are surfaces, the stability of $(-1)$-curves and the simultaneous blowing down of $(-1)$-curves have long been established (see \cites{Kod63, Iit70}). We demonstrate the good behavior of the minimal model program for families of Fano type varieties. This circle of ideas has already been applied in \cites{dFH11, HX15}. In the general setting, we establish the following deformation invariance of the MMP. 

In the theorem below, $\bc(X_U/U)$ consists of all birational contractions from $X_U/U$, and the type of an MMP refers to whether each step is divisorial, flipping, or of fiber type. See Section \ref{subsec: stable boundedness} for the detailed definition.

\begin{theorem}\label{mainthm: mmp}
Let $f: X \to T$ be a projective fibration. Suppose that there exists a Zariski dense subset $S\subset T$ such that for any $s\in S$, the fiber $X_s$ satisfies one of the cases in Theorem \ref{thm: FT is constructible}. Then, after a dominant generically finite base change of $T$, there exists a non-empty Zariski open subset $T_0 \subset T$ such that any Zariski open subset $U \subset T_0$ satisfies the following properties.
\begin{enumerate}
\item Suppose that $X'/U \in \bc(X_U/U)$. If $g: X' \dashrightarrow Y/U$ is a birational contraction with $Y$ a $\bQ$-factorial variety, then for any $t \in U$, the induced map $g_t: X'_t \dashrightarrow Y_t$ is again a birational contraction. Moreover, for any $\bR$-Cartier divisor $\mD$ on $X'$, we have
\[
(g_t)_*(\mD|_{X'_t}) \sim_{\bR} (g_*\mD)|_{Y_t}, \quad t \in U.
\]
\item For any $\bR$-divisor $\mD\in N^1(X_U/U)$, a sequence of $\mD$-MMP$/U$ induces a sequence of $\mD|_{X_t}$-MMP of the same type for each $t\in U$. 
\item Conversely, for any $\bR$-divisor $D\in N^1(X_t)$ with $t\in U$, a sequence of $D$-MMP on $X_t$ is induced by a sequence of $\mD$-MMP$/U$ on $X_U/U$ of the same type for some $\mD\in N^1(X_U/U)$. Moreover, we can choose $\mD$ such that $\mD|_{X_t}\sim_\bR D$. 
\item If $Y/U\in \bc(X_{U}/U)$, then the natural maps
\[
N^1(Y/U) \to N^1(Y_t), \quad \Nef(Y/U) \to \Nef(Y_t), \quad \NE(Y_t) \to \NE(Y/U)
\] are isomorphisms for any $t\in U$.
\end{enumerate}
\end{theorem}

As an immediate application of these results, we establish a boundedness result concerning Fano type varieties. Let $X \to T$ be a fixed family of Fano type varieties. For any closed point $t\in T$, let ${\rm bcm}(X_t)$ consist of normal projective varieties which are birational contraction models of $X_t$, that is,
\[
{\rm bcm}(X_t)\coloneqq\{Y \mid \text{there exists a birational contraction~} X_t\dashrightarrow Y \}.
\] 
Set 
\[
{\rm bcm}(X/T)\coloneqq\bigcup_{t\in T} {\rm bcm}(X_t).
\] Recall that a set of varieties is bounded if it can be parametrized by the fibers of a morphism between two schemes of finite type. We have the following result on boundedness of ${\rm bcm}(X/T)$:

\begin{theorem}\label{mainthm: bdd}
The set ${\rm bcm}(X/T)$ is bounded.
\end{theorem}

Theorem \ref{mainthm: bdd} is proved following the guiding principle that boundedness of varieties can often be deduced from birational boundedness together with the finiteness of models. This strategy has been successfully applied to many boundedness problems (see \cites{HX15, HMX18, MST20, FHS20}, etc.). The proof of Theorem~\ref{mainthm: bdd} is essentially a restatement of the argument of \cite{HX15}*{Theorem~1.3}, with additional clarification on how to obtain a fiber-wise $\bQ$-factorial modification.

Theorems~\ref{mainthm: decomposition}, \ref{mainthm: mmp}, and \ref{mainthm: bdd} are further generalized and applied in the subsequent paper \cite{CLLZ25}, where a hybrid of Mori dream spaces and Calabi-Yau type varieties (termed Morrison-Kawamata dream spaces) is investigated.

We outline the structure of the paper. Section \ref{sec: pre} provides the necessary background and sets up notation. Section \ref{sec: FT in family} is devoted to the proof of Theorem \ref{thm: FT is constructible}, which gives a partial answer to Question \ref{que: FT fiber to FT}. Section \ref{sec: deform cones} investigates the deformation invariance of N\'eron-Severi spaces and nef cones. Section \ref{sec: mmp} establishes fiber-wise small $\mathbb{Q}$-factorial modifications, the uniform behavior of the MMP for families of Fano type varieties, and the deformation invariance of effective cones and Mori chamber decompositions. Finally, Section \ref{sec: bdd FT} presents an application of the preceding results to the boundedness of birational contraction models of Fano type varieties. 

\subsection*{Acknowledgements}  
We would like to thank Christopher Hacon, Xiaowen Hu, Yifei Zhu and Ziquan Zhuang for answering our questions, and Guodu Chen for helpful discussions. We are also grateful to Stefano Filipazzi for helping us clarify the importance of taking base changes, and to Vyacheslav Shokurov for bringing our attention to his paper \cite{Sho20}. We sincerely thank the anonymous referee for many constructive suggestions, which have greatly improved the paper. S. Choi is partially supported by Samsung Science and Technology Foundation under Project Number SSTF-BA2302-03. Z. Li is partially supported by NSFC (No.12471041), the Guangdong Basic and Applied Basic Research Foundation (No.2024A1515012341), and a grant from SUSTech. C. Zhou is supported by a grant from Xiamen University (No. X2450214).

\section{Preliminaries}\label{sec: pre}

We introduce the necessary background materials. Along with this process, we fix the notation and terminology. 

A variety means an integral separated scheme of finite type over $\bC$. A point of a variety is understood to be a closed point and an open subset of a variety is meant to be a non-empty Zariski open subset, unless explicitly stated otherwise. For a variety $T$, by "shrinking $T$" we mean replacing $T$ with a Zariski open subset. A subset $S$ of a variety consists of very general points if there exist at most countably many non-empty open subsets $U_i$ such that $S = \cap U_i$.

A projective morphism $f: X \to S$ between normal varieties is called a fibration if it is surjective and $f_*\mO_X=\mO_S$. If $S' \to S$ is a morphism, then $X_{S'}$ denotes the fiber product $X \times_{S}S'$. Similarly, if $\mD$ is an $\bR$-Cartier divisor on $X$, then $\mD_{S'}$ denotes the pullback of $\mD$ to $X_{S'}$. A birational map $g: X \dasharrow X'$ between normal varieties is called a birational contraction if $g^{-1}$ does not contract any divisor.

Suppose that $\Delta \geq 0$ is an $\bR$-divisor on a normal variety $X$. Then $(X,\Delta)$ is called a log pair. A log pair $(X,\Delta)$ has klt singularities if $K_X+\Delta$ is $\bR$-Cartier where $K_X$ is the canonical divisor of $X$, and there exists a log resolution $\pi: Y \to X$ such that in the expression
\begin{equation}\label{eq: klt}
    K_Y=\pi^*(K_X+\Delta)+D,
\end{equation}
the coefficients of $D$ are greater than $-1$. Note that in \eqref{eq: klt}, $K_Y$ is chosen to be the unique Weil divisor on $Y$ such that $\pi_*K_Y=K_X$. Similarly, if the coefficients of $D$ are greater than or equal to $-1$, then $(X,\Delta)$ is said to have lc singularities. See \cite{KM98}*{\S 2.3} for more detailed discussions. 

 Let $f: X \to T$ be a projective morphism between normal varieties. We use ``$/T$" to denote properties which are relative to $T$. Then $X$ is of Fano type$/T$ if there exists a divisor $\Delta$ such that $(X, \Delta)$ has klt singularities and $-(K_X+\Delta)$ is ample$/T$. Note that in the definition, we do not assume that $K_X$ is a $\bQ$-Cartier divisor. By passing to a small $\bQ$-factorization, it is straightforward to see that in the definition of Fano type variety, $\Delta$ can be chosen to be a $\bQ$-divisor. Moreover, if $K_X$ is a $\bQ$-Cartier divisor, then $X$ is of Fano type$/T$ if and only if there exists a $\bQ$-Cartier divisor $B$ which is big over $T$ such that $(X, B)$ has klt singularities and $K_X+B \sim_{\bQ} 0/T$. Besides, $X \to T$ is called a family of Fano type varieties if all of its closed fibers are Fano type varieties. A variety $X$ is called weak Fano if $-K_X$ is nef and big. Hence, a klt weak Fano variety is automatically of Fano type. For some $n \in \mathbb{N}$, a divisor $\Delta \in |-nK_X|$ is called an lc (resp. klt) $n$-complement of $K_X$ if $(X, \frac{1}{n}\Delta)$ has lc (resp. klt) singularities (see \cite{PS09} for a more general definition). A set of varieties $\mathcal{P}$ is said to admit bounded lc (resp. klt) complements if there exists an integer $n$, depending only on $\mathcal{P}$, such that every element of $\mathcal{P}$ has an lc (resp. klt) $n$-complement.

A Cartier divisor $D$ is movable$/T$ on $X/T$ if the codimension of the support of the sheaf ${\rm coker}(f^*f_*\mathcal O_X(D) \to \mathcal O_X(D))$ is greater than or equal to $2$. 
We list relevant vector spaces and cones which appear in this paper:

\begin{enumerate}
\item $\Pic(X/T)_{\bZ}$: the abelian group generated by Cartier divisors on $X$ modulo linear equivalences over $T$.
\item $\Pic(X/T)$: the $\bR$-vector space generated by Cartier divisors on $X$ modulo linear equivalences over $T$, which is the same as $\Pic(X/T)_{\bZ}\otimes_{\bZ} \bR$ . 

\item $N^1(X/T)$: the $\bR$-vector space generated by Cartier divisors on $X$ modulo numerical equivalences over $T$. 

\item $N_1(X/T)$: the $\bR$-vector space generated by curves that are contracted by $X\to T$ modulo numerical equivalences. 
 
\item $\NE(X/T)\subset N_1(X/T)$: the cone generated by curves that are contracted by $X\to T$ (its closure $\overline{\NE}(X/T)$ is called the Mori cone).

\item (nef cone) $\Nef(X/T)\subset N^1(X/T)$: the cone generated by nef Cartier divisors on $X$ over $T$.

\item (effective cone) $\Eff(X/T)\subset N^1(X/T)$: the cone generated by effective Cartier divisors on $X$ over $T$ (its closure $\bEff(X/T)$ is called the pseudo-effective cone).

\item (movable cone) $\Mov(X/T)\subset N^1(X/T)$: the cone generated by movable Cartier divisors on $X$ over $T$. 
\end{enumerate}

If $T=\Spec (\bC)$, then we will omit $T$ in the above notations. For simplicity, when we say that $D \in N^1(X/T)$ is a divisor, we mean that $[D] \in N^1(X/T)$ with $D$ an $\bR$-Cartier divisor. This convention applies to other cones as well. We use ${\rm Int}(C)$ to denote the relative interior of the cone $C$. A cone $C\subset N^1(X/T)$ is called a rational polyhedral cone if it is generated by finitely many rational rays (hence, it is a closed cone). In general, $\Pic(X/T)$ may be an infinite dimensional vector space. However, when $X/T$ is of Fano type, it is a finite-dimensional vector space. Besides, many of the above cones are neither open nor closed. Even worse, their closures may not be generated by rational rays. In the literature, there are various conventions regarding effective and movable cones (e.g., taking the closure of the cone or selecting rational elements within the cones). However, when $X/T$ is of Fano type, $\NE(X/T), \Nef(X/T), \Eff(X/T)$ and $\Mov(X/T)$ are all closed rational polyhedral cones. 
Since this paper focuses on Fano type varieties, we do not need to distinguish these variations in the definitions of cones.

\section{Families with fibers of Fano type varieties over Zariski dense subsets}\label{sec: FT in family}

The purpose of this section is to prove Theorem \ref{thm: FT is constructible}. We begin by establishing a preparatory lemma.

\begin{lemma}\label{lem: small q-fact}
Let $X \to T$ be a projective fibration. Assume that 
\begin{enumerate}
\item there exists a Zariski dense subset $S \subset T$ such that $X_s$ is of Fano type for each $s\in S$, and
\item there exists $\Delta\geq 0$ such that $(X, \Delta)$ is lc and $K_X+\Delta\sim_{\bQ}0/T$.
\end{enumerate}
Then, after shrinking $T$, there exists a small $\bQ$-factorial modification $Z \to X$ such that $Z$ has klt singularities.
\end{lemma}
\begin{proof}
Let $h: W \to X$ be a $\bQ$-factorial dlt modification of $(X, \Delta)$ (see \cite{BCHM10}) such that
\[
K_W+\Delta_W+\Exc(h)=h^*(K_X+\Delta),
\] where $\Delta_W$ is the strict transform of $\Delta$, and $\Exc(h)$ is the union of $h$-exceptional divisors (with coefficient $1$). Hence, any $h$-exceptional divisor is an lc place of $(X, \Delta)$. Run a $(K_W+\Exc(h))$-MMP/$X$, 
\[
\phi: W \dasharrow Z/X,
\] which terminates at $Z$ with $K_Z+\phi_*\Exc(h)$ nef over $X$ (see \cite{HX13}). We claim that $Z \to X$ is a small morphism after shrinking $T$. 

This is straightforward if $K_X$ is $\bQ$-Cartier: as $X_s$ is of Fano type for any $s\in S$, we see that $X_s$ has klt singularities (see \cite{Laz04b}*{Theorem 9.5.19}). Therefore, $X$ still has klt singularities after shrinking $T$. This means that the lc centers of $(X, \Delta)$ are contained in $\Supp(\Delta)$. Hence, we have $\Delta_W+E=h^*\Delta$ with $\Supp(E) = \Exc(h)$. This implies that
\[
K_W+\Exc(h)=h^*K_X+E,
\] and any $(K_W+\Exc(h))$-MMP/$X$ will contract $E$. In particular, $Z \to X$ is a small morphism. In the sequel, we explain that the same property still holds even if $K_X$ is not $\bQ$-Cartier.

First, after shrinking $T$, we still have $h(\Exc(h)) \subset \Supp(\Delta)$. This is because $K_X$ is $\bQ$-Cartier on $X \backslash \Supp(\Delta)$.  As for any $s\in S$, the fiber $X_s$ is of Fano type with $S$ a Zariski dense subset, we see that $X \backslash \Supp(\Delta)$ is klt after shrinking $T$. This implies that $h(\Exc(h)) \subset \Supp(\Delta)$. 

Suppose that we have $\phi_*\Exc(h) \neq 0$. As $h(\Exc(h))\subset  \Supp(\Delta)$, we have $q(\phi_*\Exc(h)) \subset \Supp(\Delta)$, where $q: Z \to X$. We can cut $X$ by sufficiently ample and general hypersurfaces $H_i \subset X$ for $i=1,\ldots,l$ until 
\[
q\left(\phi_*\Exc(h)\right) \cap H_1\cap \cdots \cap H_l
\] consists of just points. We use the same notation to denote the varieties and the divisors after this cutting. Note that as these hypersurfaces are sufficiently ample, $K_Z+\phi_*\Exc(h)$ is still nef over $X$. Choose any $x_0 \in q\left(\phi_*\Exc(h)\right) \cap H_1\cap \cdots \cap H_l$. Then $q^{-1}(x_0)$ is of positive dimension. To be precise, we write 
\[
q^{-1}(x_0) = F \cup F',
\] where $F$ consists of $q$-exceptional divisors, and $F'$ consists of higher codimensional subvarieties. Let $\Delta_Z = \phi_*\Delta_W$. As $x_0 \in q(\phi_*\Exc(h)) \subset \Supp(\Delta)$, and $q$ has connected fibers, we see that 
\[
q^{-1}(x_0) \cap \Supp(\Delta_Z) \neq \emptyset.
\] Thus, there exists a curve $\ell \subset q^{-1}(x_0)$ such that 
\[
\Supp(\Delta_Z) \cap \ell \neq \emptyset \quad\text{and} \quad \ell \not\subset \Supp(\Delta_Z).
\] This is obvious if $F' \not\subset \Supp(\Delta_Z)$. If $F' \subset \Supp(\Delta_Z)$, then as $\Supp(\Delta_Z)$ does not contain $q$-exceptional divisors and $q^{-1}(x_0)$ is connected, we can find a curve $\ell \subset F$ satisfying the desired property. However, as
\[
K_Z+\Delta_Z+\phi_*\Exc(h)=q^*(K_X+\Delta),
\] and $K_Z+\phi_*\Exc(h)$ is nef$/X$, we have
\[
0<(K_Z +\phi_*\Exc(h))\cdot \ell+\Delta_Z \cdot \ell = q^*(K_X+\Delta) \cdot \ell =0.
\] This is a contradiction. Therefore, $Z \to X$ is a small morphism. Moreover, $Z$ has klt singularities as $(Z, \phi_*\Exc(h))$ has dlt singularities. This shows that $Z \to X$ is a small $\bQ$-factorial modification of $X$ and $Z$ has klt singularities. 
\end{proof}

We can now prove Theorem \ref{thm: FT is constructible} which is a partial answer to Question \ref{que: FT fiber to FT}.

\begin{proof}[Proof of Theorem \ref{thm: FT is constructible}]
We proceed with the argument in several steps. In Steps 1 and 2, we apply uniform modifications to all three cases and then treat each case separately. In what follows, after shrinking $T$, we still use $S$ to denote the Zariski dense subset $S \cap T$.

Step 1. As for any $s\in S$, the fiber $X_s$ is a normal variety, it is regular in codimension $1$ and satisfies Serre's condition $S_2$. Shrinking $T$, we can assume that $X$ is regular in codimension $1$. By \cite{EGAIV}*{(9.9.2)(viii)},  the set of points (not necessarily closed points)
\[
E_2 \coloneqq \{p \in S \mid \mO_X|_{X_p} \text{~is~} S_2\}
\]
is a constructible set.  As $S$ is Zariski dense and $S\subset E_2$, the generic point of $T$ lies in $E_2$. Therefore, after shrinking $T$, we can assume that $X$ is $S_2$ and $T$ is normal. This shows that $X$ is normal, and thus $f$ is a projective contraction.

Step 2. In this step, we show that in each of the three cases, $X$ can be further assumed to be $\bQ$-factorial with klt singularities. More precisely, assuming that there exists a Zariski dense subset $S \subset T$ such that $X_s$ is of Fano type for each $s \in S$ (which is automatically satisfied in the three cases), we prove that $X$ admits a small $\bQ$-factorial modification with klt singularities after shrinking $T$. Note that in this step, we only need to assume that $S$ is Zariski dense.

First, we claim that there exist an $m\in \bZ_{>0}$ independent of $s\in S$ and a Weil divisor $B^s \in |-mK_{X_s}|$ such that $(X_s, \frac 1 m B^s)$ has lc singularities for any $s\in S$. As $X_s$ is of Fano type, we can replace it with a small $\bQ$-factorial modification $Y_s \to X_s$. Then $Y_s$ is still of Fano type, and it suffices to show the claim for $Y_s$. Therefore, we can assume that $X_s$ is $\bQ$-factorial. Running a $(-K_{X_s})$-MMP, $X_s \dasharrow X'_s$, we have that $X'_s$ is a Fano type variety with $-K_{X'_s}$ nef. By \cite{Birkar19}*{Theorem 1.7}, there exist an $m\in \bZ_{>0}$  independent of $s\in S$ and a Weil divisor $D^s \in |-mK_{X'_s}|$ such that $(X'_s, \frac 1 m D^s)$ has lc singularities. Let $p: W \to X_s, q: W \to X'_s$ be a resolution of $X_s \dasharrow X'_s$.  Then we have
\[
mp^*K_{X_s} =mq^*K_{X_s'}-F
\] for some $q$-exceptional divisor $F \geq 0$. Hence, we have
\[
mp^*K_{X_s} +F+q^*D^s = q^*(mK_{X_s'}+D^s)\sim 0.
\] In particular, $F+q^*D^s$ is an effective Weil divisor. Set
\[
B^s= p_*(F+q^*D^s).
\] We have $B^s \in |-mK_{X_s}|$. As
\[
p^*(mK_{X_s}+B^s)=q^*(mK_{X_s'}+D^s)\sim 0,
\] we see that $(X_s, \frac 1 m B^s)$ still has lc singularities.

Next, we show that there exists a divisor $B$ such that $(X, \frac 1 m B)$ is lc and $mK_X+ B \sim 0/T$. Consider the coherent sheaf $\mO_X(-mK_X)$. By generic flatness, we can assume that $\mO_X(-mK_X)$ is flat over $T$ after shrinking $T$. Replacing $T$ by an open affine subset, we can assume that
\[
H^0(T, f_*\mO_X(-mK_X)) \to H^0(X_t, \mO_X(-mK_X)|_{X_t})
\] is surjective for any $t\in T$. This uses \cite{Har77}*{III, Theorem 12.8, Corollary 12.9}, and the fact that the coherent sheaf $f_*\mO_X(-mK_X)$ is globally generated over an affine set. Shrinking $T$ further, we have $\mO_X(-mK_X)|_{X_t} = \mO_{X_t}(-mK_{X_t})$ for each $t\in T$. In summary, we obtain a surjection
\begin{equation}\label{eq: surj}
H^0(T, f_*\mO_X(-mK_X)) \to H^0(X_t, \mO_{X_t}(-mK_{X_t})), \quad t\in T.
\end{equation} 

For any $s\in S$, there exists $B^s\in |-mK_{X_s}|$ such that
$(X_s, \frac 1 m B^s)$ is lc. By \eqref{eq: surj}, choose
$B\in |-mK_X|$ mapping to $B^s$. Then $(X, \frac 1 m B)$ is lc in a Zariski neighborhood of $X_s$ (see \cite{Laz04b}*{Theorem 9.5.19}). Shrinking $T$, we have that $(X, \frac 1 m B)$ is lc and $mK_X+ B \sim 0/T$. This shows the first part of the claim.

To complete the remaining part of the claim, we explain that it suffices to assume that $X$ is a $\bQ$-factorial klt variety. By Lemma \ref{lem: small q-fact}, after shrinking $T$, $X$ admits a small $\bQ$-factorial modification $W \to X$ such that $W$ has klt singularities. Shrinking $T$ further, we may assume that $W_t \to X_t$ is a small modification for each $t \in T$. In particular, the three cases remain valid after replacing $X$ with $W$. Moreover, if $W$ is of Fano type over $T$, then $X$ is of Fano type over $T$. Replacing $X$ by $W$, we can assume that $X$ is a $\bQ$-factorial klt variety. 

Step 3. In this step, we show the claim in Case (ii). Just as \eqref{eq: surj}, for each $m\in \bN$, there exists a non-empty open affine subset $U_m\subset T$ such that
\[
H^0(U_m, f_*\mO_X(-mK_X)) \to H^0(X_t, \mO_{X_t}(-mK_{X_t}))
\] is surjective for any $t\in U_m$. As $S$ consists of very general points, we have $S \cap (\cap_{m\in \bN} U_m) \neq \emptyset$. Therefore, if $s\in S \cap (\cap_{m\in \bN} U_m)$, then for any $m\in \bN$, there exists an open subset $U_m$ such that the natural map
\[
H^0(U_m, f_*\mO_X(-mK_X)) \to H^0(X_t, \mO_{X_s}(-mK_{X_s}))
\] is surjective. As $X_s$ is of Fano type, there exists some $m \in \bN$ and a divisor $D \in |-mK_{X_s}|$ such that $(X_s, \frac{1}{m} D)$ has klt singularities. By the above surjection, there exists a divisor $\mD \in H^0(X_{U_m}, \mO_X(-mK_X))$ such that $\mD|_{X_s}=D$. As $(X_s, \frac{1}{m} D)$ is klt, $(X, \frac 1 m \mD)$ is klt in a Zariski neighborhood of $X_s$ (see \cite{Laz04b}*{Theorem 9.5.19}). Shrinking $T$, we can assume that $(X, \frac 1 m \mD)$ is klt. As $X$ is $\bQ$-factorial and $X_s$ is of Fano type, we have that $\mD_s$ is big for any $s\in S$. As $S$ consists of very general points, $\mD$ is big over $T$ (see \cite{Li22}*{Theorem 3.18}). Therefore, $X$ is of Fano type over $T$.

Step 4. In this step, we show that Case (i) can be reduced to Case (iii). Shrinking $T$, we can assume that $K_X|_{X_s}=K_{X_s}$. As $X$ is $\bQ$-factorial, there exists $n\in \bN$ such that $nK_{X_s}$ is Cartier for each $s\in S$. As $-K_{X_s}$ is nef and big, by the effective base-point free theorem (\cite{Kol93}), there exists $m\in \bN$ such that $|-mK_{X_s}|$ is base-point free for each $s\in S$. By Bertini theorem, a general element $B \in |-mK_{X_s}|$ is such that the pair $(X_s, \frac{1}{m} B)$ has klt singularities. This shows that $\{X_s \mid s\in S\}$ admits bounded klt complements.

Step 5. In this step, we establish the claim for Case (iii), which consequently implies the result for Case (i) as well.

Recall that $X$ is a $\bQ$-factorial variety with klt singularities by Step 2. By assumption, there exists an integer $n\in \bN$, such that for any $s\in S$, the fiber $X_s$ admits a klt $n$-complement. As in \eqref{eq: surj}, possibly shrinking $T$, the natural restriction map
\[
H^0(X, \mO_{X}(-nK_X)) \to H^0(X_s, \mO_{X_s}(-nK_{X_s}))
\] is surjective for each $s\in S$. Therefore, if we have $\Delta_s\in  |-nK_{X_s}|$ such that $(X_s, \frac 1 n \Delta_s)$ has klt singularities, then there exists $\Delta\in |-nK_X|$ such that $(X, \frac 1 n \Delta)$ has klt singularities in a Zariski neighborhood of $X_s$ (see \cite{Laz04b}*{Theorem 9.5.19}). Shrinking $T$, we can assume that $(X, \frac 1 n \Delta)$ has klt singularities with $K_X+\frac 1 n \Delta\sim_\bQ 0$. Hence, there exists $\epsilon \in \bQ_{>0}$ such that $(X, (\frac 1 n +\epsilon)\Delta)$ still has klt singularities. By \cite{HMX18}*{Theorem 4.2}, after shrinking $T$, we know that the Kodaira dimension $\kappa(X_t, K_{X_t}+(\frac 1 n +\epsilon)\Delta_t)$ is independent of $t\in T$. As $K_X+(\frac 1 n +\epsilon)\Delta \sim_\bQ \epsilon n(-K_X)$, we see that $-K_X$ is big over $T$. This implies that $X$ is of Fano type over $T$.
\end{proof}

In analogy to Question \ref{que: FT fiber to FT}, we ask:
\begin{question}
Let $f: X \to T$ be a projective fibration. Suppose that there exists a Zariski dense subset $S \subset T$ such that $X_s$ is a Mori dream space for each $s \in S$. Is it true that $X$ is a Mori dream space over $T$ after shrinking $T$?
\end{question}

\section{Deformation of N\'eron-Severi spaces and nef cones}\label{sec: deform cones}

In this section, we study the deformation of N\'eron-Severi spaces and various cones for a family of varieties under certain conditions.

\subsection{Deformation of N\'eron-Severi spaces}\label{subsec: deform NS space}

Let $f: X \to T$ be a projective morphism such that $f_*\mO_X =\mO_T$. Let ${\mP}ic_{X/T}$ be the sheaf associated to the relative Picard functor
\[
S \mapsto \Pic(X_S)_{\bZ}/\Pic(S)_{\bZ}=\Pic(X_S/S)_{\bZ},
\] where $S\subset T$ is a Zariski open subset. See \cite{Kle05}*{\S 9.2} for details. Note that ${\mP}ic_{X/T}$ is denoted by $\Pic_{(X/T)(\rm{zar})}$ in \cite{Kle05}*{Definition 9.2.2}. In general, ${\mP}ic_{X/T}(U)$ may not be $\Pic(X_U/U)_{\bZ}$ for an open subset $U\subset T$ because of the sheafification. In fact, by \cite{Kle05}*{(9.2.11.2)}, we always have
\begin{equation}\label{eq: R1}
{\mP}ic_{X/T}(U) = H^0(U, R^1f_*\mO^*_{X_U}).
\end{equation} If $X \to T$ is a flat morphism, then by \cite{Kle05}*{(9.2.11.3)}, we have the exact sequence
\[
0 \to \Pic(U)_{\bZ} \to \Pic(X_U)_{\bZ} \to {\mP}ic_{X/T}(U). 
\] In particular, we have the inclusion 
\begin{equation}\label{eq: inclusion of Pic to functor}
\Pic(X_U/U)_{\bZ} \hookrightarrow {\mP}ic_{X/T}(U).
\end{equation}

In the following, we show that if $X\to T$ is a family of Fano type varieties, then there exists a dominant generically finite base change $T' \to T$ and an open subset $U\subset T'$, such that ${\mP}ic_{X_U/U} \otimes \bR$ is a constant sheaf. To be precise, this means that there exist natural isomorphisms
\[
{\mP}ic_{X_U/U}(V) \otimes \bR \simeq  \Pic(X_V/V) \simeq N^1(X_V/V)
\] for any open subset $V\subset U$. Recall that by our convention, $\Pic(X/T)$ is the $\bR$-vector space generated by Cartier divisors on $X$ modulo linear equivalences over $T$, which is the same as $\Pic(X/T)_{\bZ}\otimes_{\bZ} \bR$. In fact, only certain vanishing properties are needed to ensure the statement holds. Theorem \ref{thm: deform pic} provides a more general formulation.

We need several lemmas in the proof of Theorem \ref{thm: deform pic}.

First, recall the definition of the triviality of a map (see \cite{Mat22}*{\S 11} for details). Consider a diagram of spaces and mappings in the analytic topology:

\[
\begin{tikzcd}
X \arrow[rr, "f"] \arrow[dr, "\pi_1"'] & & Y \arrow[dl, "\pi_2"] \\
& Z. &
\end{tikzcd}
\]
We say that $f$ is trivial over $Z$ if there exist spaces $X_0$ and $Y_0$, a mapping $f_0: X_0 \to Y_0$, and homeomorphisms 
\[
X \cong X_0 \times Z, \quad Y \cong Y_0 \times Z
\]
such that the following diagram is commutative:
\[
\begin{tikzcd}
X \arrow[rr, "f"] \arrow[dd, "\cong"'] \arrow[dr, "\pi_1"] & & Y \arrow[dd, "\cong"] \arrow[dl, "\pi_2", swap] \\
& Z &\\
X_0 \times Z \arrow[rr, "f \times \mathrm{id}_Z"'] \arrow[ru] && Y_0 \times Z. \arrow[ul]
\end{tikzcd}
\]
We say that $f$ is locally trivial over $Z$ if, for any point $z \in Z$, the above property holds in some analytic neighborhood of $z$.

The following local triviality of the morphism is likely well-known to experts. For completeness, we explain how it can be deduced from the literature. 

\begin{lemma}\label{lem: trivial morph}
    Let $h:Y \to X$ and $X \to T$ be projective fibrations of varieties. Then there exists a Zariski open subset $U \subset T$ such that $h_U: Y_U \to X_U$ is locally trivial over $U$ (in the analytic topology).
\end{lemma}
\begin{proof}
    After shrinking $T$, we can assume that there exists a Whitney-Thom stratification of $h$ (in the Zariski topology), which satisfies Thom's  $A_h$-condition (see \cite{Hir77}*{Page 247, Theorem 2} or the proof of \cite{Hir77}*{Page 248, Corollary 1}). Shrinking $T$ further, by Thom’s second isotopy lemma (see \cite{Mat22}*{Proposition 11.1}), we can assume that the morphism $h$ is locally trivial over $T$. 
\end{proof}

\begin{lemma}\label{lem: trivial div}
Let $g: Y \to X$ be a projective morphism between normal varieties with $R^1g_{*}\mO_{Y}=0$. If $D$ is a $\mathbb{Q}$-Cartier divisor on $Y$ with $D \equiv 0/X$, then $D\sim_\bQ 0/X$.

In particular, suppose that $g$ is a birational morphism and that $(X, \Delta)$ is a klt pair for some effective $\mathbb{R}$-divisor $\Delta$ on $X$. If $D$ is a $\mathbb{Q}$-Cartier divisor on $Y$ with $D \equiv 0/X$, then $g_*D$ is a $\mathbb{Q}$-Cartier divisor on $X$ such that $g^*(g_*D) = D$.
\end{lemma}

\begin{proof}
The argument below closely follows the proof of \cite{KM92}*{Proposition (12.1.4)} and strengthens it. For the reader’s convenience, we include the full proof.

Let $x\in X$ be a point. Then there exists a triangulation of the topological space $Y$ such that the analytic subset $g^{-1}(x)$ appears as the support of a sub-complex (see \cite{BHPV04}*{I, Theorem (8.8)}). Let $x\in U$ be a sufficiently small Stein neighborhood. Then $g^{-1}(U)$ deformation retracts onto the sub-complex $g^{-1}(x)$ (see \cite{Hat00}*{Appendix, Proposition A.5}). In particular, we have
\[
H^i(g^{-1}(U), \bZ)=H^i(g^{-1}(x), \bZ), \quad i \in \bZ_{\geq 0}.
\] Let $(-)^{\rm an}$ denote the analytification functor. As $R^1g^{\rm an}_*\mO^{\rm an}_Y=(R^1g_*\mO_Y)^{\rm an}=0$ (see \cite{SGA1}*{\textsc{Expos\'e} XII, Th\'eor\`eme 4.2}) and $H^1(U,\mO^{\rm an}_X)=0$, we have
\[
H^1(g^{-1}(U), \mO^{\rm an}_Y)=0.
\] By the exponential sequence, we have
\[
H^1(g^{-1}(U), \mO^{\rm an}_Y) \to H^1(g^{-1}(U),\mO_Y^{\rm an,*}) \to H^2(g^{-1}(U),\bZ),
\] and hence we obtain the natural injective group homomorphism
\[
{\Pic}(g^{-1}(U)) \hookrightarrow H^2(g^{-1}(U),\bZ)=H^2(g^{-1}(x),\bZ).
\] If $D \equiv 0/X$, then $D|_{g^{-1}(x)} \equiv 0$ and thus $D|_{g^{-1}(U)}$ is a torsion element of $H^2(g^{-1}(x),\bZ)$. If $d$ is the order of the torsion group of $H^2(g^{-1}(x),\bZ)$, then $dD|_{g^{-1}(U)}=0 \in \Pic(g^{-1}(U))$. Note that there are only finitely many topological types of the fibers of $g$ (for example, using Lemma \ref{lem: trivial morph}). We can take $d\in \bN$ sufficiently divisible, such that for each $x'\in X$, there exists an analytic neighborhood $x'\in U'$ such that $dD|_{g^{-1}{(U')}} \sim 0$ on $g^{-1}{(U')}$. As $g^{\rm an}_*\mO^{\rm an}_Y(dD) \simeq (g_*\mO_Y(dD))^{\rm an}$ (see \cite{SGA1}*{\textsc{Expos\'e} XII, Th\'eor\`eme 4.2}), we have that $g_*\mO_Y(dD)$ is a line bundle on $X$. By the natural sheaf morphism, $g^*g_*(\mO_Y(dD)) \to \mO_Y(dD)$, we have the natural sheaf morphism (in the Zariski topology)
\[
\theta: \mO_Y \to \mO_Y(dD) \otimes (g^*g_*(\mO_Y(dD)))^{\vee}.
\] On $g^{-1}(U')$, as $g^{\rm an,*}g^{\rm an}_*(\mO^{\rm an}_Y(dD)) \to \mO^{\rm an}_Y(dD)$ is surjective and $\mO^{\rm an}_Y$ is torsion free, we see that $\theta^{\rm an}|_{g^{-1}(U')}$ is an isomorphism. Thus $\theta^{\rm an}$ is an isomorphism. This implies that $\theta$ is an isomorphism (for example see \cite{SGA1}*{\textsc{Expos\'e} XII, Th\'eor\`eme et d\'efinition 1.1}). This shows $D \sim_{\bQ} 0/X$.

Next, assume that $g$ is birational and there exists a divisor $\Delta$ such that $(X, \Delta)$ has klt singularities. Replacing $Y$ by a resolution, we can assume that $Y$ is smooth. As klt singularities are rational singularities, we have $R^1g_*\mO_Y=0$. Hence, we have $D \sim_\bQ 0/X$ by the previous discussion. This implies that $g_*D$ is a $\bQ$-Cartier divisor on $X$ and $D=g^*(g_*D)$ by the negativity lemma. This completes the proof.

Finally, we include an alternative proof of this special case below, which is more algebraic in nature.

Let $h: \tilde X \to X$ be a small $\bQ$-factorial modification. Then we have
    \[
    K_{\tilde X}+\tilde \Delta=h^*(K_X+\Delta),
    \] where $\tilde \Delta$ is the strict transform of $\Delta$. Replacing $Y$ by a higher model, we can assume that $g$ factors through $h$, and let $u: Y \to \tilde X$ be the corresponding morphism. We have $D \equiv 0/\tilde X$. Thus $\tilde B \coloneqq u_*D$ is a $\bQ$-Cartier divisor as $\tilde X$ is $\bQ$-factorial and $D =u^*(u_*D)$ by the negativity lemma. In particular, we have $\tilde B\equiv 0/X$. By the base-point free theorem, $\tilde B$ is semi-ample over $X$. Hence $g_*D=h_*\tilde B$ is a $\bQ$-Cartier divisor on $X$ and $h^*(h_*\tilde B)=\tilde B$. Thus, we have $g^*(g_*D)=u^*(h^*(h_*\tilde B))=D$.
\end{proof}

\begin{remark}
    Lemma \ref{lem: trivial div} does not hold for an arbitrary birational contraction $g$. Assume that $X$ is a normal projective surface which is not $\bQ$-factorial. Let $H$ be a Weil divisor on $X$ which is not $\bQ$-Cartier. For a resolution $g: Y \to X$, one can define the Mumford pull-back of $H$, denoted by $g^*H$. Then we have $g^*H \equiv 0/X$ by construction. However, we do not have $g^*H \sim_{\bQ} 0/X$, since this would imply that $H$ is $\bQ$-Cartier. Lemma \ref{lem: trivial div} is consistent with the fact that for projective surfaces, rational singularities are $\mathbb{Q}$-factorial (see \cite{Bad01}*{Theorem 4.6}).
\end{remark}

\begin{proof}[Proof of Theorem \ref{thm: deform pic}]
We first prove Theorem \ref{thm: deform pic} (1), and then show that Theorem \ref{thm: deform pic} (2) can be deduced from Theorem \ref{thm: deform pic} (1). From now on, we assume that the natural restriction map 
\begin{equation}\label{eq: surj of natural map}
 N^1(X/T) \to N^1(X_t), \quad [D] \mapsto [D|_{X_t}]   
\end{equation}
is surjective for very general $t\in T$.

Shrinking $T$, by \cite{Har77}*{III, Theorem 12.8, Corollary 12.9}, we can assume that
\begin{equation}\label{eq: R1,2}
R^1f_*\mO_X=R^2f_*\mO_X=0 \text{~and~} H^1(X_t, \mO_{X_t})=H^2(X_t, \mO_{X_t})=0~ \forall ~t\in T.
\end{equation} Shrinking $T$ further, we can assume that $T$ is smooth. Then, we have 
\begin{equation}\label{eq: Pic=N}
\Pic(X_U/U) =N^1(X_U/U) \text{~for any open~}U \subset T \text{~and~}
\Pic(X_t)=N^1(X_t) ~\forall ~t\in T.
\end{equation} 
Indeed, $\Pic(X_t)=N^1(X_t)$ follows from $H^1(X_t, \mO_{X_t})=0$. We explain that $\Pic(X_U/U) =N^1(X_U/U)$. Let $\eta\in U$ be the generic point. Then, by the flat base change, we have $H^1(X_\eta, \mO_{X_\eta})=0$. In particular, we have $\Pic(X_\eta)=N^1(X_\eta)$. Thus, if $D$ is a Cartier divisor such that $D \equiv 0/U$, then we have $D_\eta \sim_\bQ 0$. Replacing $D$ by a multiple, we can assume that $D_\eta \sim 0$, and thus there exists $\alpha\in K(X_\eta)=K(X_U)$ such that $D_\eta = {\rm div}(\alpha)$ on $X_\eta$. This implies that $D-{\rm div}(\alpha)$ is a vertical divisor on $X_U/U$. As $D-{\rm div}(\alpha) \equiv 0/U$ and $U$ is smooth, there exists a divisor $L$ on $U$ such that $D-{\rm div}(\alpha)=f^*L$ by the negativity lemma. That is, $D \sim_\bQ 0/U$. This shows $\Pic(X_U/U) =N^1(X_U/U)$. 

Let 
\begin{equation}\label{eq: flat}
\mB_i,\quad 1 \leq i \leq l
\end{equation} be Cartier divisors on $X$ such that $\{[\mB_i] \mid 1 \leq i \leq l\}$ spans the vector space $N^1(X/T)$. Shrinking $T$, we can assume that each irreducible component of $\mB_i$ is flat over $T$. By \eqref{eq: surj of natural map}, we can further assume that for very general $t\in T$, any divisor $B_t$ of $N^1(X_t)$ is a restriction of a divisor $\mB$ of $N^1(X/T)$, and each irreducible component of $\mB$ is flat over $T$.

By \eqref{eq: inclusion of Pic to functor}, we have
\[
\Pic(X_U/U)_{\bZ} \subset {\mP}ic_{X_T/T}(U)
\] for any open subset $U\subset T$.
Taking direct image sheaves of the short exact sequence, 
\[
0 \to \bZ \to \mO_{X_T} \to \mO^*_{X_T} \to 0,
\] we have the long exact sequence
\[
\cdots \to R^1f_* \mO_{X_T} \to R^1f_*\mO^*_{X_T} \to R^2f_*\bZ \to R^2f_* \mO_{X_T} \to \cdots.
\] By \eqref{eq: R1,2}, we have $R^1f_* \mO_{X_T}= R^2f_* \mO_{X_T}=0$. Therefore, we have $R^1f_*\mO^*_{X_T} \simeq R^2f_*\bZ$, and thus
\begin{equation}\label{eq: Pic =H^0}
{\mP}ic_{X_T/T}(U) = H^0(U, R^1f_*\mO^*_{X_T}) = H^0(U, R^2f_*\bZ)
\end{equation} for any open subset $U\subset T$ by equation \eqref{eq: R1}.

By the Thom-Whitney stratification and Thom’s first isotopy lemma, there exists a constructible stratification in the Zariski topology such that each stratum is locally trivial in the analytic topology (see \cite{ZS10}*{Theorem 3.2}). In particular, $R^2f_*\bZ$ restricted to each stratum is a local system in the analytic topology (see \cite{ZS10}*{Proposition 3.5}). That is, $R^2f_*\bZ$ restricting to each stratum is a locally constant sheaf in the analytic topology. Shrinking $T$, we can assume that 
$R^2f_*\bZ$ is a local system over $T$.

The key to the argument below is to show that, in the above setting, $R^2f_*\bR = (R^2f_*\bZ) \otimes \bR$ is indeed a constant sheaf on $T$ in the Zariski topology. The argument in this part is similar to the argument of \cite{dFH11}*{Lemma 6.6}.

By \eqref{eq: Pic=N}, \eqref{eq: inclusion of Pic to functor} and \eqref{eq: Pic =H^0}, we obtain the relations
\begin{equation}\label{eq: inclusions}
N^1(X_U/U)=\Pic(X_U/U) \subset  {\mP}ic_{X_T/T}(U) \otimes \bR= H^0(U, R^2f_*\bR) \subset H^0(\bD, R^2f_*\bR)
\end{equation} for any open subset $U\subset T$ and any analytic open subset $\bD \subset U$. We will show that these inclusions are indeed equalities.
By assumption (see \eqref{eq: surj of natural map}), the restriction map 
\[
\Pic(X_U/U)\to \Pic(X_t)
\] is surjective for very general $t\in U$. Let $t_0\in T$ be an arbitrary point. As $R^2f_*\bZ$ is a local system, there exists a contractible analytic open subset $t_0\in \bD\subset T$ such that
\begin{equation}\label{eq: identify}
H^0(\bD, R^2f_*\bZ) \simeq H^2(X_{t_0}, \bZ).
\end{equation} By \eqref{eq: R1,2} and \eqref{eq: Pic=N}, we have
\begin{equation}\label{eq: eq on fibers}
\Pic(X_t)_{\bZ} = H^1(X_t, \mO_{X_t}^*)=H^2(X_{t}, \bZ).
\end{equation} Therefore, for any Cartier divisor $B_{t_0}\in \Pic(X_{t_0})$, there exists a Cartier divisor $B_t \in \Pic(X_t)$ for any $t\in \bD$ through the natural identification \eqref{eq: identify}. By the density of very general points in the analytic topology, we can take $t$ to be a very general point such that $\Pic(X_U/U)|_{X_t}\simeq \Pic(X_t)$. Therefore, by the choices made in \eqref{eq: flat}, there exists a $\bQ$-Cartier divisor $\mB$ on $X_U$, which is a flat cycle over $U$, such that $\mB_{t}\coloneqq \mB|_{X_t}=B_t$. Moreover, as $\mB$ is a flat cycle over $U$,  $\mB_{t_0}$ can be identified with $\mB_{t}=B_t$ as cycles through the isomorphism
\begin{equation}\label{eq: identify fibers}
H^2(X_{t_0}, \bZ)\simeq H^0(\bD, R^2f_*\bZ) \simeq H^2(X_{t}, \bZ).
\end{equation} Hence, we have $\mB_{t_0}= B_{t_0}$ as cycles by the choice of $B_t$. We conclude that $\mB_{t_0}=B_{t_0}$ as divisors by the equality $\Pic(X_{t_0})_{\mathbb Z}=H^2(X_{t_0},\mathbb Z)$ in \eqref{eq: eq on fibers}. As $\mB \in \Pic(X_U/U)$, from \eqref{eq: identify}, we see that \eqref{eq: inclusions} are indeed equalities. Moreover, as $t_0$ is an arbitrary point, the above discussion also shows that the restriction map
\[
N^1(X_U/U) \to N^1(X_t)
\] is surjective for any $t\in U$. By the above construction, the variety $T$ obtained by shrinking the original base satisfies the conclusion of Theorem \ref{thm: deform pic} (1). We denote it by $T_0$.

Next, we deduce Theorem \ref{thm: deform pic} (2) from Theorem \ref{thm: deform pic} (1). We proceed with the argument in two steps.

Step 1. We prove the claim under the additional assumption that $X$ is smooth. We claim that there exists a dominant generically finite base change $T' \to T$ and an open subset $T_0 \subset T'$ such that 
\[
N^1(X_{T_0}/T_0) \to N^1(X_{T_0,t})
\] is surjective for very general $t\in T_0$. This can be proved by the same argument as \cite{KM92}*{Proposition 12.2.5}. Indeed, let $\mH$ be the Hilbert scheme that parametrizes $(n-1)$-cycles of fibers of $f$, where $\dim (X/T)=n$. Let $\mH = \cup_{i\in J} \mH_i$ be the decomposition into irreducible components. Then $J$ is at most a countable set and the natural map $p_i: \mH_i \to T$ is projective. Let $I = \{i \in J \mid p_i(\mH_i) \neq T\}$. Then 
\[
W = T -\bigcup_{i\in I}p_i(\mH_i)
\] consists of very general points. 

Shrinking $T$, we can assume that $X \to T$ is a locally trivial family in the analytic topology. As $X$ is smooth, shrinking $T$ further, we can assume that $f$ is a smooth morphism. Hence, $h^{1,1}(X_{t}, \bQ)$ is invariant for $t\in T$. 
%as $h^{1}(X_{t}, \mO_{X_{t}})=0$ and $f$ is smooth by assumption.
Fix some $t_0\in W$. By the choice of $W$, for each prime divisor $D$ on $X_{t_0}$, there exists some $i\in I$ and a Weil divisor $\mD$ on $X_{\mH_i}\coloneqq \mH_i \times_T X$ such that $\mD|_{X_{\mH_i,s_0}} = D$, where $s_0 \in \mH_i$ is a closed point such that $p_i(s_0)=t_0$. Note that we have $X_{\mH_i,s_0} \simeq X_{t_0}$. Let $\mG_i \subset \mH_i$ be a closed subvariety such that $\mG_i \to T$ is a dominant generically finite morphism. Then we still have $\mD|_{X_{\mG_i,s_0}} = D$, where $s_0 \in \mG_i$ is a closed point such that $p_i(s_0)=t_0$. As $f$ is smooth, $f_{\mG_i}$ is still a smooth morphism. Shrinking $\mG_i$, we can assume that $\mG_i$ is a smooth variety, and thus $X_{\mG_i}$ is a smooth variety. Hence $\mD|_{X_{\mG_i}}$ is a Cartier divisor on $X_{\mG_i}$. 

As $H^{1,1}(X_{t_0}, \bQ)$ is of finite dimension, we can repeat the above process finitely many times, and obtain a dominant generically finite morphism $q: T' \to T$ and an open subset $T_0 \subset T'$ such that 
\[
N^1(X_{T_0}/T_0) \to N^1(X_{T_0,t_0})
\] is surjective. Moreover, each Cartier divisor $B$ on $X_{T_0,t_0}$ is a restriction of a cycle $\mB$ on $X_{T_0}$ which is flat over $T_0$. Since $X \to T$ is a locally trivial family in the analytic topology, there exists a contractible analytic open set $\bD \subset T$ containing $t_0$ such that $X_{\bD}$ is trivial over $\bD$. By the flatness of $\mB$ over $T_0$, for each $t\in \bD$, $[\mB_t]$ corresponds to $[\mB_{t_0}]$ under the identification $H^2(X_{t_0}, \bQ) \simeq H^2(X_{t}, \bQ)$. By the connectedness of $T_0$, we see that
\[
N^1(X_{T_0}/T_0) \to N^1(X_{T_0,t})
\] is surjective for each $t\in q^{-1}(W)\cap T_0$. This verifies the conditions in Theorem \ref{thm: deform pic} (1). Note that for the dense subset $S' = q^{-1}(S)$, each point $s \in S'$ has the property that the fiber $X_{T_0, s}$ has rational singularities and $H^i(X_{T_0, s}, \mO_{X_{T_0, s}}) = 0, i = 1, 2$. The above shows that Theorem \ref{thm: deform pic} (2) holds when $X$ is smooth.

Step 2. We show the claim in the general case. 

Let $h: Y \to X$ be a resolution which can be assumed to be a fiber-wise resolution after shrinking $T$. Note that for the dense subset $S$, each point $s \in S$ has the property that the fiber $Y_s$ has rational singularities and $H^i(Y_s, \mO_{Y_s}) = 0, i = 1, 2$. By Step 1, there exist a dominant generically finite base change $q: T' \to T$ and an open subset $T_0 \subset T'$, such that $N^1(Y_{T_0}/T_0) \to N^1(Y_{T_0,t})$ is surjective for $t\in q^{-1}(W)\cap T_0$, where $W\subset T$ consists of very general points. Moreover, each Cartier divisor $B$ on $Y_{T_0,t_0}$ is a restriction of a cycle $\mB$ on $Y_{T_0}$ which is flat over $T_0$. 

\iffalse 
Shrinking $T_0$, we can assume that $T_0 \to q(T_0)$ is an \'etale morphism. Hence there also exists a divisor $\Delta'$ on $X_{T_0}$ such that $(X_{T_0}, \Delta')$ has klt singularities (see \cite{KM98}*{Proposition 5.20}). 
\fi

The above morphisms are summarized in the following diagram:
\[
	\begin{tikzcd}
	 Y_{T'}\arrow[d,"h_{T'}"] \arrow[r] &Y \arrow[d, "h"] \\
	 X_{T'} \arrow[d,"f_{T'}"] \arrow[r]  &X \arrow[d, "f"]\\
     T' \arrow[r,"q"] & T.
	\end{tikzcd}
\] 

As before, after shrinking $T$, we can assume that $f$ is locally trivial in the analytic topology. Moreover, by Lemma \ref{lem: trivial morph}, the morphism $h$ can also be locally trivialized over $T$ in the analytic topology. In particular, this implies that for a fixed $t_0\in T$, there exists an analytic open subset  $t_0\in \mathbb D \subset T$, such that for any $t\in \mathbb D$, the following diagram commutes
\[
 	\begin{tikzcd}
	 H^2(Y_{t_0}, \bZ)\arrow[r, "\simeq"] &H^2(Y_{t}, \bZ) \\
	 H^2(X_{t_0}, \bZ)\arrow[u, "h_{t_0}^*"]\arrow[r, "\simeq"] &H^2(X_{t}, \bZ) \arrow[u, "h_{t}^*", swap],
	\end{tikzcd}   
\]
where $H^2(Y_{t_0}, \bZ) \simeq H^2(Y_{t}, \bZ)$ and $H^2(X_{t_0}, \bZ)\simeq H^2(X_{t}, \bZ)$ are obtained through natural identifications as \eqref{eq: identify fibers}. 

Take any $t_0 \in W \cap q(T_0)$. Note that by \eqref{eq: R1,2} and \eqref{eq: Pic=N}, $\dim N^1(X_{t})$ is invariant for each $t\in T$. Let 
\[
h_{t_0}^*(N^1(X_{t_0}))\coloneqq \{[h_{t_0}^*B] \mid [B] \in N^1(X_{t_0})\}
\] be the subspace of $N^1(Y_{t_0})$. Let $s_0\in q^{-1}(W)\cap T_0$ be a closed point such that $q(s_0)=t_0$. By the discussion in the last paragraph of Step 1, there are finitely many flat cycles $\mB_i, 1 \leq i \leq l$ on $Y_{T_0}$ over $T_0$ such that $\{[\mB_{i,s_0}]\in N^1(Y_{T_0, s_0})\}$ generates the subspace $h^*(N^1(X_{t_0}))$. Because $Y_{T_0}$ is locally trivial over $T_0$ in the analytic topology, we see that $\{[\mB_{i,s}] \mid 1 \leq i \leq l\}$ spans a subspace of the same dimension as $\dim N^1(X_{t_0})$. 

In the analytic topology, as $f$ is locally trivial and $h$ is locally trivial over $T$, we have that $f\circ h: Y \to T$ is also locally trivial. The corresponding property is preserved after the base change $T' \to T$. In particular, there exists an analytic open subset $s_0\in \bD \subset T_0$ such that for any $s\in \mathbb D$, the following diagram commutes
\begin{equation}\label{eq: trivial over T'}
	\begin{tikzcd}
	 H^2(Y_{T_0,s_0}, \bZ)\arrow[r, "\simeq"] &H^2(Y_{T_0,s}, \bZ) \\
	 H^2(X_{T_0,s_0}, \bZ)\arrow[u, "h_{T_0, t_0}^*"]\arrow[r, "\simeq"] &H^2(X_{T_0,s}, \bZ) \arrow[u, "h_{T_0, s}^*", swap],
	\end{tikzcd}
\end{equation} where $H^2(Y_{T_0, s_0}, \bZ) \simeq H^2(Y_{T_0, s}, \bZ)$ and $H^2(X_{T_0, s_0}, \bZ)\simeq H^2(X_{T_0, s}, \bZ)$ are obtained through natural identifications as in  \eqref{eq: identify fibers}. As $\mB_{i,s_0} =h_{T_0,s_0}^*B$ for some $\bQ$-Cartier divisor $B$ on $X_{T_0,s_0}$ by construction, we have $\mB_{i,s_0} \equiv 0/X_{T_0,s_0}$. Thus, we have $\mB_{i,s} \equiv 0/X_{T_0,s}$ for $s\in \bD$ by \eqref{eq: trivial over T'}. By the connectedness of $T_0$, we have
\[
\mB_i \equiv 0/X_{T_0}.
\] 
As general fibers of $f_{T_0}$ have rational singularities, after shrinking $T_0$, we can assume that $X_{T_0}$ still has rational singularities (see \cite{Elk78}*{Théorème 2}). By Lemma \ref{lem: trivial div}, we have $\mB_i \sim_\bQ 0/X_{T_0}$. As $h_{T_0}$ is a birational morphism, $\mD_i\coloneqq (h_{T_0})_*\mB_i$ is a $\bQ$-Cartier divisor and $\mB_i=h_{T_0}^*\mD_i$ by the negativity lemma. Therefore, we have 
\[
h^*_{T_0,s_0}(\mD_{i,s_0})\sim_\bQ h^*_{T_0,s_0}(B),
\] and hence $\mD_{i,s_0}\sim_\bQ B$. This shows that
\[
N^1(X_{T_0}/T_0) \to N^1(X_{T_0,s_0})
\] is surjective. 

Thus, $\{[\mD_{i,s}] \mid 1 \leq i \leq l\}$ spans a subspace of the same dimension as $\dim N^1(X_{t_0})$. Hence, the natural map 
\[
N^1(X_{T_0}/T_0) \to N^1(X_{T_0,s})
\] is surjective for each $s\in q^{-1}(W)\cap T_0$. This verifies that the conditions of Theorem \ref{thm: deform pic} (1) are satisfied, and thus the desired claim follows.
\end{proof}

\begin{corollary}\label{cor: FT and rationally connected}
Let $f: X \to T$ be a projective fibration. Suppose that $S\subset T$ is a Zariski dense subset such that for any $s\in S$, the fiber $X_s$ is a rationally chain-connected variety with rational singularities. Then, up to a dominant generically finite base change of $T$, there  exists a non-empty open subset $T_0\subset T$, such that ${\mP}ic_{X_{T_0}/T_0} \otimes \bR$ is a constant sheaf. Moreover, for any open subset $U\subset T_0$, the natural restriction maps
\[
N^1(X_{T_0}/T_0) \to N^1(X_U/U) \to N^1(X_t)
\] are isomorphisms for any $t\in U$. In particular, if for any $s\in S$, the fiber $X_s$ is a variety of Fano type, then the above claim holds. 
\end{corollary}
\begin{proof}
Shrinking $T$, we have $H^1(X_s, \mO_{X_s})=H^2(X_s, \mO_{X_s})=0$ for each $s\in S$ as $X_s$ is a rationally chain-connected variety with rational singularities (see \cite{Kol96}*{IV, Theorem 3.10, Corollary 3.8}). Hence, the claim follows from Theorem \ref{thm: deform pic}. Besides, a Fano type variety is naturally a rationally chain-connected variety with rational singularities.
\end{proof}

\subsection{Deformation of nef cones}\label{subsec: nef cone}

\begin{theorem}\label{thm: nef cone}
Let $f: X \to T$ be a projective fibration. Assume that $\Nef(X_t)$ is a rational polyhedral cone for each $t\in T$. Suppose that for any open subset $V\subset T$, 
\begin{enumerate}
\item the natural restriction map $N^1(X_V/V) \to N^1(X_t), t\in V$ is an isomorphism, and
\item for any $\mD \in \bEff(X_V/V)$, there exists a $\mD$-MMP$/V$ which terminates with a model such that the push-forward of $\mD$ is nef.
\end{enumerate}
Then there exists a non-empty Zariski open subset $T_0\subset T$, such that for any open subset $U\subset T_0$, we have
\[
\Nef(X_{T_0}/T_0) \simeq \Nef(X_U/U) \simeq \Nef(X_t), t\in U
\] under the natural restriction maps. Moreover, there are isomorphisms of Mori cones
\[
\NE(X_t) \simeq \NE(X_U/U) \simeq \NE(X_{T_0}/T_0), t\in U
\] under the natural inclusion maps.

\end{theorem}

\begin{proof}
We proceed with the argument in several steps.

Step 1. If we have $\mD \in N^1(X/T)$ such that $\mD_t \coloneqq \mD|_{X_t}$ is nef, then $\mD_t$ is nef for very general $t\in T$ by the openness of ampleness. Set 
\[
\mN_\mD \coloneqq \{t\in T \mid \mD_t \text{~ is nef~}\}.
\] We claim that $\mN_\mD$ is an open subset. By assumption, we can run some $\mD$-MMP/$T$, $X\dasharrow X'$, such that $\mD'$ is nef$/T$, where $\mD'$ is the strict transform of $\mD$ on $X'$. This MMP is invariant when restricting to the fiber $X_t$ for any $t\in \mN_\mD$. As the image $Z \subset T$ of the contracting/flipping loci is a closed subset of $T$, and since $Z = T \setminus \mN_{\mD}$, it follows that $\mN_{\mD}$ is an open subset of $T$.

Note that $\Nef(X_{t})$ is a rational polyhedral cone by assumption. Suppose that $\Nef(X_{t})$ is generated by finitely many rays $\mD_i$ for $i=1, \ldots, l$. Then, the fact that $\mN_{\mD_i}$ is open for each $i$ implies that the map $t\mapsto \Nef(X_t)$ is lower semi-continuous in the following sense: for any $t_0\in T$, there exists a Zariski open set $U_{t_0}$ containing $t_0$ such that $\Nef(X_{t})$ contains $\Nef(X_{t_0})$ for any $t\in U_{t_0}$. Here, we view $\Nef(X_t)$ as a subset of $N^1(X/T)\cong N^1(X_t)$ by assumption, and the cones are compared under this natural identification. By abuse of notation, the inclusion ``$\subset$" is understood under this identification.

Step 2. We show that there exists an open subset $T_0\subset T$ such that $\Nef(X_t)$ is invariant for all $t \in T_0$.  We argue by induction on the dimension of the base $T$. The case when $\dim T=0$ is obvious. Assume that the claim holds when $\dim T =n-1$. Then we show that it also holds when $\dim T=n$.

First, we choose countably many hypersurfaces $\{Z_i\}$ in $T$ such that $\cup Z_i$ is Zariski dense in $T$. For each $Z_i$, by induction, there exists an open subset $U_i\subset T$ such that $U_i\cap Z_i$ is non-empty and $\Nef(X_t)$ is invariant on $U_i\cap Z_i$. We set $P_i\coloneqq\Nef(X_{t_i})$ for some $t_i\in U_i\cap Z_i$. By the lower semi-continuity, after shrinking $U_i$ around $t_i$, we can assume that $P_i \subset \Nef(X_t)$ for any $t\in U_i$. By abuse of notation, we also write $P_i$ for the cone inside $N^1(X/T)$ under the natural identification $N^1(X/T) \simeq N^1(X_{t_i})$. Let 
\[
P_\infty\coloneqq\overline{\cup_i P_i}
\] be the closure of $\cup_i P_i$ inside $N^1(X/T)$. Set $U_\infty\coloneqq\cap_i U_i$. We claim that $P_\infty=\Nef(X_t)$ for any $t\in U_\infty$. First, we have $P_\infty\subset \Nef(X_t)$ for any $t\in U_\infty$ as $P_i \subset \Nef(X_t)$ for $t\in U_i$. Suppose that $P_\infty\subsetneqq \Nef(X_{t_0})$ for some $t_0\in U_\infty$, then, by the lower semi-continuity, there exists an open subset $t_0\in V\subset T$ such that $P_\infty\subsetneqq \Nef(X_t)$ for any $t\in V$. Since $\cup_i Z_i$ is dense in $T$, $V$ must intersect some $Z_i\cap U_i$. Thus, we have
\[
P_\infty\subsetneqq \Nef(X_t)
\] for some $t\in U_i\cap V$. However, we have $\Nef(X_t)= P_i$ for any $t\in U_i \cap Z_i$ by construction. This is a contradiction. Therefore, we have $P_\infty=\Nef(X_t)$ for any $t\in U_\infty$.

For any $t'\in U_\infty$, there exists an open subset $t'\in V_0\subset T$ such that $P_\infty=\Nef(X_{t'}) \subset \Nef(X_t)$ for each $t\in V_0$. If there exists some $\tilde t$ such that the inclusion is strict, then, by the lower semi-continuity, there exists an open subset $\tilde t \in \tilde V_0$ such that $P_\infty \subsetneqq \Nef(X_t)$ for any $t\in \tilde V_0$. However, as $U_\infty$ consists of very general points, we have $U_\infty \cap \tilde V_0 \neq \emptyset$. This is a contradiction as $\Nef(X_{t''})=P_\infty$ for any $t''\in U_\infty$ as shown above. This shows that $P_\infty= \Nef(X_t)$ for each $t\in V_0$. Therefore, the natural restriction map
\[
\Nef(X_{V_0}/V_0) \to \Nef(X_t), t\in V_0
\] is an isomorphism. In fact, if $\mD\in N^1(X_{V_0}/V_0)$ is a divisor such that $\mD_t\in \Nef(X_t)$, then $\mD_{t'} \in N^1(X_{t'})$ is the divisor identified with $\mD_t$ for each $t'\in V_0$. As $ \Nef(X_t)=P_\infty= \Nef(X_{t'})$ for $t, t'\in V_0$, we see that $\mD_{t'} \in \Nef(X_{t'})$. In particular, we have $\mD \in \Nef(X_{V_0}/V_0)$.

Denote $T_0:=V_0$. For any open subset $U\subset T_0$, we have the natural inclusions
\[
\Nef(X_{T_0}/T_0)|_{X_t} \hookrightarrow \Nef(X_U/U)|_{X_t} \hookrightarrow \Nef(X_t)
\] for any $t\in U$. As $\Nef(X_{T_0}/T_0)|_{X_t}= \Nef(X_t)$, we have $\Nef(X_U/U)|_{X_t} = \Nef(X_t)$. As the natural restriction map $N^1(X_U/U) \to N^1(X_t)$ is an isomorphism, we have $\Nef(X_U/U) \simeq \Nef(X_t)$ for any $t\in U$.

Step 3. For the last statement on the isomorphism of Mori cones, the natural maps
\[
\NE(X_t) \to \NE(X_U/U) \to \NE(X_{T_0}/T_0), t\in U
\] are injective as $N^1(X_{T_0}/T_0) \simeq N^1(X_U/U) \simeq N^1(X_t)$. Thus, we have $\NE(X_t) \subset \NE(X_U/U)$ under this identification. If the inclusion is strict, then there exists a divisor $\mD$ on $X_U/U$ such that $\mD$ is not nef$/U$ while $\mD_t \in \Nef(X_t)$. By $\Nef(X_U/U) \simeq \Nef(X_t)$, we see that $\mD$ is nef$/U$. This is a contradiction. The remaining isomorphisms can be shown similarly.
\end{proof}

As a corollary of Theorem \ref{thm: nef cone}, we have the following two corollaries.

\begin{corollary}\label{cor: nef cone for MDS}
Let $f: X \to T$ be a projective fibration. Assume that $X_t$ is a Mori dream space for each $t\in T$. Suppose that for any open subset $V\subset T$, 
\begin{enumerate}
\item the natural restriction map $N^1(X_V/V) \to N^1(X_t)$ for any $t\in V$ is an isomorphism, and
\item $X_V$ is a Mori dream space over $V$.
\end{enumerate}
Then there exists a non-empty Zariski open subset $T_0\subset T$, such that for any open subset $U\subset T_0$, we have
\[
\Nef(X_{T_0}/T_0) \simeq \Nef(X_U/U) \simeq \Nef(X_t), t\in U
\] under the natural restriction maps. Moreover, there are isomorphisms of Mori cones
\[
\NE(X_t) \simeq \NE(X_U/U) \simeq \NE(X_{T_0}/T_0), t\in U
\] under the natural inclusion maps.
\end{corollary}
\begin{proof}
As $X_t$ is a Mori dream space, $\Nef(X_t)$ is a rational polyhedral cone. As $X_V$ is a Mori dream space over $V$, for any pseudo-effective divisor $\mD$ on $X_V$ over $V$, there exists a $\mD$-MMP$/V$ that terminates with a model such that the push-forward of $\mD$ is semi-ample. This verifies the conditions in Theorem \ref{thm: nef cone}, and thus the claim follows.
\end{proof}

\begin{corollary}\label{cor: nef cone for FT}
Let $f: X \to T$ be a projective fibration. Suppose that there exists a Zariski dense subset $S\subset T$ such that for any $s\in S$, the fiber $X_s$ satisfies one of the cases in Theorem \ref{thm: FT is constructible}. Then, up to a dominant generically finite base change of $T$, there exists a non-empty Zariski open subset $T_0\subset T$, such that for any open subset $U\subset T_0$, we have
\[
\Nef(X_{T_0}/T_0) \simeq \Nef(X_U/U) \simeq \Nef(X_t), t\in U
\] under the natural restriction maps. Moreover, there are isomorphisms of Mori cones
\[
\NE(X_t) \simeq \NE(X_U/U) \simeq \NE(X_{T_0}/T_0), t\in U
\] under the natural inclusion maps.

\end{corollary}
\begin{proof}
By Theorem \ref{thm: FT is constructible}, we can assume that $X$ is of Fano type over $T$ after shrinking $T$. Shrinking $T$ further, we can assume that $X_t$ is of Fano type for each $t\in T$. Note that a variety of Fano type (resp. over $T$) is a Mori dream space (resp. over $T$). Taking a dominant generically finite base change of $T$ and shrinking $T$ again, we may assume that Corollary \ref{cor: nef cone for MDS} (1) is satisfied by Theorem \ref{thm: deform pic} (2). Therefore, the claim follows from Corollary \ref{cor: nef cone for MDS}.
\end{proof}

\section{MMP in a family and deformations of effective cones and Mori chamber decompositions}\label{sec: mmp}

This section is devoted to the proof of Theorem \ref{thm: fiberwise q-fact}, Theorem \ref{mainthm: mmp}, as well as the invariance of effective cones and Mori chamber decompositions stated in Theorem \ref{mainthm: decomposition}.

\subsection{Stable boundedness}\label{subsec: stable boundedness}
Let $X \to T$ be a projective fibration. Suppose that $X$ is of Fano type over $T$. Set 
\[
\begin{aligned}
\bc(X/T) \coloneqq \{h \mid {} &h: X \dasharrow Y/T \text{~is a birational contraction}/T\\
&\text{~up to isomorphisms of~} Y/T\}.
\end{aligned}
\] Note that the above set includes not only varieties birationally contracted by $X$ but also the birational contraction maps from $X$. However, for the sake of simplicity, we also write $Y/T \in \bc(X/T)$ instead of $[h: X \dasharrow Y/T]\in \bc(X/T)$. For $Y_1/T, Y_2/T \in \bc(X/T)$, $Y_1 \simeq Y_2/T$ should be understood to mean that there exists an isomorphism $\theta: Y_1 \to Y_2$, such that $h_2 =\theta \circ h_1: X \dasharrow Y_2$, where $h_i: X \dasharrow Y_i, i=1,2$ are given birational contractions. Moreover, for a subvariety $V\subset T$, we write $Y_V/V$ instead of the base change of $h$ to $V$. 

To study an MMP of a Fano type variety $X/T$, we also consider a subset $\bc(X/T)_{\rm MMP}$ of $\bc(X/T)$ which we define below. First, we call a sequence of birational contractions
\[
X=X_0 \dto X_1 \dto \cdots \dto X_{n}/T
\] a partial MMP$/T$ (with respect to a divisor $D$) if each step consists of $D_i$-extremal divisorial or flipping contractions and $D_i$-flips, where $D_i$ is the strict transform of $D$ on $X_i$. To be precise, if $X_i \to X_{i+1}/T$ is a $D_i$-flipping contraction, then $X_{i+2} \to X_{i+1}/T$ is its $D_i$-flip. Besides, in the above partial MMP, $D_n$ is not required to be nef over $T$. On the other hand, if $D_{n}$ becomes nef and big over $T$, and $X_{n} \to X_{n+1}$ is the birational contraction induced by $D_n$ (in this case, the relative Picard number $\rho(X_n/X_{n+1})$ may be greater than $1$), then
\[
X=X_0 \dto X_1 \dto \cdots \dto X_{n}\to X_{n+1}/T
\] is also called a partial MMP$/T$. If a fibration $X \to Z$ has relative Picard number $1$, then it is called an extremal contraction.

We set
\begin{equation}\label{eq: partial MMP}
\begin{split}
    \bc(X/T)_{\rm MMP}\coloneqq \{h \mid &X \dasharrow Y/T \text{~is a partial~} \text{MMP}/T \\
    &\text{up to isomorphisms of~} Y/T \}.
\end{split}
\end{equation} 
Later, we will show that $\bc(X/T)_{\rm MMP}=\bc(X/T)$ when $X$ is $\bQ$-factorial in Proposition \ref{prop: bc=bcmmp}.

Let $\tilde X \to X$ be a small $\bQ$-factorial modification. Since $\tilde X$ and $X$ are isomorphic in codimension $1$, $\bc(\tilde X/T)$ can be naturally identified with $\bc(X/T)$. By abuse of notation, we will write $\bc(\tilde X/T)=\bc(X/T)$ for this identification.  The following is a special case of \cite{BCHM10}*{Corollary 1.1.5}.

\begin{proposition}\label{prop: finiteness}
Let $X$ be of Fano type over $T$. Then $\bc(X/T)$ is a finite set.
\end{proposition}

In fact, we have the following stronger statement. This proposition is a special case of \cite{CLLZ25}*{Proposition 6.11}. For the reader's convenience, we include the proof.

\begin{proposition}\label{prop: strong finiteness}
Let $X$ be of Fano type over $T$. Then for any open subset $V\subset T$ and $Y/V\in \bc(X_V/V)$, there exists an element $W/T\in \bc(X/T)$ such that $W_V/V$ is isomorphic to $Y/V$.
\end{proposition}

\begin{proof}
Replacing $X$ by a small $\bQ$-factorial modification, we can assume that $X$ is $\bQ$-factorial. Let $h: X_V \dto Y/V$ be the corresponding birational contraction. Let $A$ be an ample divisor on $Y$ over $V$. Let $A'$ be the strict transform of $A$ on $X_V$ and let $B$ be the Zariski closure of $A'$ on $X$. Note that $B_V\coloneqq B|_{X_V}=A'$ and $B$ is a $\bQ$-Cartier divisor as $X$ is $\bQ$-factorial. Let $p: \Gamma \to X_V, q: \Gamma \to Y$ be birational morphisms such that $h=q\circ p^{-1}$. Then we have
\begin{equation}\label{eq: compare pullback}
p^*B_V=q^*A+E,
\end{equation} where $E$ is an effective $p$-exceptional divisor. As $X$ is of Fano type over $T$, there exists a divisor $D$ such that $(X,D)$ has klt singularities and $K_X+D \sim_\bR 0/T$. Hence, there exists $\epsilon>0$ such that $(X,D+\epsilon B)$ still has klt singularities. Thus, we have
\[
K_X+D+\epsilon B \sim_\bR \epsilon B/T.
\]
By \eqref{eq: compare pullback}, $h: X_V \dto Y/V$ is the canonical model of $(X_V/V, D_V+\epsilon B_V)$ over $V$. As $X/T$ is of Fano type, $(X/T, D+\epsilon B)$ admits the canonical model $\theta: X \dto W/T$ over $T$. Since $W/T \in \bc(X/T)$, by the uniqueness of the canonical model, we have $\theta|_{X_V} = h$ by construction.
\end{proof}

\subsection{MMP in a family}
In this subsection, we prove Theorem \ref{thm: fiberwise q-fact} and Theorem \ref{mainthm: mmp}. Recall that
\[
\bc(X/T)_{\rm MMP}\coloneqq \{h \mid X \dasharrow Y/T \text{~is a partial~}\text{MMP}/T \text{~up to isomorphisms of~} Y/T\}
\] is a subset of $\bc(X/T)$.

\begin{lemma}\label{lem: surj preserved for bir contractions}
Let $X \to T$ be a projective fibration. 
\begin{enumerate}
    \item Let $X \dashrightarrow Y/T$ be a divisorial contraction, a flipping contraction, or a flip. Suppose that the natural map $N^1(X/T) \to N^1(X_t)$ is surjective for very general $t \in T$. Then there exists an open subset $T_0\subset T$ such that the natural map
\[
N^1(Y_{T_0}/T_0) \to N^1(Y_t)
\]
is surjective for very general $t \in T_0$.
    \item Assume that $X$ is of Fano type over $T$. Suppose that the natural map $N^1(X/T) \to N^1(X_t)$ is surjective for very general $t\in T$. Then there exists an open subset $T_0 \subset T$ such that for any open subset $V\subset T_0$, if $Y/V \in  \bc(X_V/V)_{\rm MMP}$, then the natural map
    \[
    N^1(Y/V) \to N^1(Y_t)
    \] is surjective for very general $t\in V$.
\end{enumerate}
\end{lemma}
\begin{proof}
We divide the proof into several steps. 

Step 1. First, we prove statement (1). Shrinking $T$, we may assume that $X \dto Y/T$ is not a constant map over the generic point of $T$. If $g: X \to X_1/T$ is a birational extremal contraction (divisorial or flipping contraction), then we have the commutative diagram  \[
	\begin{tikzcd}
	 N^1(X/T) \arrow[d]&N^1(X_1/T)\arrow[d]\arrow[l, "g^*",swap] \\
	 N^1(X_t) &N^1(X_{1,t})\arrow[l, "g_t^*",swap].
	\end{tikzcd}
\] For a very general $t$, if $B$ is a $\bQ$-Cartier divisor on $X_{1,t}$, then $g_t^*B$ is a $\bQ$-Cartier divisor on $X_t$ and thus there exists a $\bQ$-Cartier divisor $D$ on $X$ such that $D|_{X_t} \equiv g_t^*B$. In particular, we have $D \equiv 0/X_1$ as the relative Picard number $\rho(X/X_1)=1$. This implies that there exists a $\bQ$-Cartier divisor $D'$ on $X_1$ such that $D=g^*D'$. Thus we have $D'|_{X_{1,t}} \equiv B$.

Let $g\colon X \to X_1/T$ be a flipping contraction with the flip $g^+\colon X^+ \to X_1/T$. Assume that $H$ is a Cartier divisor on $X$ which is anti-ample over $X_1$. We claim that, for very general $t\in T$, the induced morphism $g_t\colon X_t \to X_{1,t}$ is also an $H_t$-flipping contraction. Here, $g_t$ and $H_t$ denote the restrictions to the fiber $X_t$ of $g$ and $H$, respectively. We use the same convention below. As $g_t$ is a small morphism, it suffices to show $\rho(X_t/X_{1,t})=1$. If $B$ is a $\bQ$-Cartier divisor on $X_t$, then we can assume that there exists a $\bQ$-Cartier divisor $D$ on $X$ such that $D_t \equiv B$ by the above argument. As $\rho(X/X_1)=1$, we have $D-cH \equiv 0/X_1$ for some $c\in \bQ$. In particular, there exists a $\bQ$-Cartier divisor $D'$ on $X_1$ such that $D-cH =g^*D'$. Thus we have $D_t-cH_t =g_t^*D'_t$ for the $\bQ$-Cartier divisor $D'_t$ on $X_{1,t}$ which shows the desired claim. As $g_t^+$ is also a small morphism with $H_t^+$ ample over $X_{1,t}$, we see that $g_t^+: X^+_t \to X_{1,t}$ is the flip of $g_t$. If $B^+$ is a $\bQ$-Cartier divisor on $X_t^+$, then there exists $r \in \bQ$ such that $B^+ -rH^+_t \equiv 0/X_{1,t}$. Hence, there exists a $\bQ$-Cartier divisor $B_1$ on $X_{1,t}$ such that $(g_t^{+})^*B_1=B^+ -rH^+_t$. By the previous argument, there exists a $\bQ$-Cartier divisor $D_1$ on $X_1$ such that $D_{1,t} \equiv B_1$. If $D^+=(g^+)^*D_1$, then we have $D^+_t+rH^+_t\equiv B^+$. This shows the claim.

    Step 2. In this step, we explain that (2) can be reduced to a fixed MMP.
    
    By Proposition \ref{prop: strong finiteness}, for any open subset $V\subset T$ and $Y/V \in \bc(X_V/V)$, there exists some $Y'/T \in \bc(X/T)$ such that $Y \simeq Y'_V/V$. Consider all $Y'/T \in \bc(X/T)$ such that there exists an open subset $U \subset T$ with $Y'_U \in \bc(X_U/U)_{\mathrm{MMP}}$. Since $\bc(X/T)$ is a finite set by Proposition \ref{prop: finiteness}, after shrinking $T$, we may assume that for any open subset $V \subset T$ and any $Y/V \in \bc(X_V/V)_{\mathrm{MMP}}$, there exists some $Y'/T \in \bc(X/T)_{\mathrm{MMP}}$ such that $Y \simeq Y'_V/V$.

    For $V\subset T_0 \subset T$, by the natural maps 
    \[
    N^1(Y'_{T_0}/T_0) \to N^1(Y_V/V) \to N^1(Y_t), t\in V, 
    \] it suffices to find an open subset $T_0\subset T$ such that for each $Y'/T\in \bc(X/T)_{\rm MMP}$, the natural map $ N^1(Y'_{T_0}/T_0) \to N^1(Y_t)$ is surjective for very general $t\in T_0$. As $\bc(X/T)_{\rm MMP}$ is a finite set, it suffices to show this property for a fixed $Y'/T \in \bc(X/T)_{\rm MMP}$.

Step 3. In this step, we show the statement (2). 

If $X \dto Y'$ consists only of divisorial contractions, flipping contractions, and flips, then the claim follows from statement (1).

However, in our definition of a partial MMP, the last step may be a birational contraction whose relative Picard number is greater than $1$. Suppose that $h\colon Z \to Y'/T$ is such a birational contraction between varieties of Fano type over $T$, and that the natural map $N^1(Z/T)\to N^1(Z_t)$ is surjective for very general $t\in T$. Note that $h$ corresponds to the contraction of curves belonging to a face $F$ of the Mori cone $\NE(Z/Y')$. We can contract an extremal ray of $F$ to obtain a morphism $h': Z \to Z'$ and thus $h$ factors through $h'$. By repeating this process, we can further assume that $h$ is an extremal contraction. Then the desired property follows from Step 1. 
\end{proof}

Using Lemma \ref{lem: surj preserved for bir contractions}, we now prove the claim concerning fiber-wise small $\bQ$-factorial modifications.

\begin{proof}[Proof of Theorem \ref{thm: fiberwise q-fact}]
    Let $h: Y \to X$ be a log resolution of $(X, \Delta)$. Shrinking $T$, we can assume that $Y \to T$ is a smooth morphism. As for any $s\in S$, the fiber $X_s$ has rational singularities with $h^{i}(X_s, \mO_{X_s})=0$ for $i=1,2$, we can assume that the same property still holds for $Y_s$ with $s\in S$ after shrinking $T$. By Theorem \ref{thm: deform pic} (2), there exists a dominant generically finite base change $T' \to T$  and an open subset $T'_0\subset T'$ such that $N^1(X_{T_0'}/T_0') \to N^1(X_{t'})$ is surjective for each $t' \in T_0'$. Shrinking $T_0'$, we can assume that $T_0'$ is a smooth variety. Thus $Y_{T'_0}$ is still a smooth variety. Applying Theorem \ref{thm: deform pic} (2) to $Y_{T'_0} \to T'_0$, there exists a dominant generically finite base change $T'' \to {T'_0}$ and an open subset $T_0'' \subset T''$ such that $N^1(Y_{T_0''} / T_0'') \to N^1(Y_{t''})$ is surjective for each $t'' \in T_0''$. Note that the map $N^1(X_{T''_0}/T''_0) \to N^1(X_{t''})$ remains surjective for each $t''\in T''_0$, since this property is preserved under further base change. Shrinking $T_0''$, we can assume that there still exists a $\bQ$-divisor $\Delta''$ such that $(X_{T_0''}, \Delta'')$ has klt singularities. Thus, we can replace $T''_0$ by $T$ and assume that
    \[
    N^1(X/T) \to N^1(X_{t}) \text{~and~} N^1(Y/T) \to N^1(Y_{t})
    \] are surjective for each $t\in T$.

Now we proceed with the standard process of small $\bQ$-factorial modifications. Let $K_Y+ \Delta_Y = h^*(K_X+\Delta)+E$ such that $\Supp E = \Supp \Exc(h)$ and $(Y, \Delta_Y)$ has klt singularities. Shrinking $T$, we can assume that $h$ is a fiber-wise resolution such that the same property of $K_{Y_t}+ \Delta_{Y_t} = h_t^*(K_{X_t}+\Delta_t)+E_t$ holds fiber-wise, where $\Delta_{Y_t}=(\Delta_Y )|_{Y_t}$. By \cite{BCHM10}, there exists a $(K_Y+\Delta_Y)$-MMP over $X$ which terminates and only $E$ is contracted in this process. Denote this MMP$/X$ by 
\[
g: Y=Y_0 \dto Y_1 \dto \cdots \dto Y_n=W /X.
\] Moreover, $g$ only consists of divisorial contractions, flipping contractions and flips. Shrinking $T$, we can assume that if $Y_j \to Y_{j+1}$ (resp. $Y_{j+1} \to Y_{j}$) is a divisorial or a small contraction (resp. small contraction), then for each $t\in T$, $Y_{j,t} \to Y_{j+1,t}$ (resp. $Y_{j+1,t} \to Y_{j,t}$) is still a divisorial or a small contraction (resp. small contraction).

Note that $W \to X$ is a small $\bQ$-factorial modification of $X$. We claim that it is also a small $\bQ$-factorial modification of each fiber after possibly shrinking $T$ again.

By Lemma \ref{lem: surj preserved for bir contractions} (1), after shrinking $T$, we have a surjection $N^1(Y_j/T) \to N^1(Y_{j,t})$ for $0 \leq j \leq n$ over a very general point $t\in T$. By Theorem \ref{thm: deform pic} (1), after shrinking $T$, we can assume that $N^1(Y_j/T) \to N^1(Y_{j, t})$ is an isomorphism for each $t\in T$ and $0 \leq j \leq n$. 

Suppose that $Y_j \to Y_{j+1}$ is a $(K_{Y_j}+\Delta_{Y_j})$-negative extremal contraction (divisorial or flipping contraction). By Lemma \ref{lem: trivial div}, after shrinking $T$, we have the short exact sequences
\[
\begin{split}
0 \to &N^1(Y_{j+1}/T) \to N^1(Y_j/T) \to N^1(Y_j/Y_{j+1}) \to 0, ~ 0 \leq j \leq n\\
    0 \to &N^1(Y_{j+1, t}) \to N^1(Y_{j,t}) \to N^1(Y_{j,t}/Y_{j+1, t}) \to 0,~t\in T.
\end{split}
\] By the Five Lemma, the commutativity of the following diagram 
\[
\begin{tikzcd}
	0\arrow[r] & N^1(Y_{j+1}/T) \arrow[d, "\simeq"] \arrow[r] &N^1(Y_j/T)\arrow[r] \arrow[d, "\simeq"]& N^1(Y_j/Y_{j+1}) \arrow[r] & 0 \\
	 0\arrow[r] & N^1(Y_{j+1, t}) \arrow[r] &N^1(Y_{j,t})\arrow[r] & N^1(Y_{j,t}/Y_{j+1,t}) \arrow[r] & 0.
	\end{tikzcd}
\] implies that $N^1(Y_j/Y_{j+1}) \simeq N^1(Y_{j,t}/Y_{j+1,t})$. In particular, $Y_{j,t} \to Y_{j+1,t}$ is still a $(K_{Y_{j,t}}+\Delta_{Y_{j,t}})$-negative extremal contraction. Hence, $Y_{j,t} \to Y_{j+1,t}$ is a divisorial (resp. flipping) contraction if and only if $Y_{j} \to Y_{j+1}$ is a divisorial (resp. flipping) contraction. 

Suppose that $Y_j \to Y_{j+1}$ is a $(K_{Y_j}+\Delta_{Y_j})$-negative flipping contraction with flip $Y_{j+2} \to Y_{j+1}$. As $(K_{Y_{j+2,t}}+\Delta_{Y_{j+2,t}})$ is ample over $Y_{j+1}$, we see that $Y_{j+2,t} \to Y_{j+1,t}$ is the flip of the flipping contraction $Y_{j,t} \to Y_{j+1,t}$.

The above shows that for each $t\in T$, 
\[
g_t: Y_t=Y_{0,t} \dto Y_{1,t} \dto \cdots \dto Y_{n,t}=W_t /X_t
\] is also a $(K_{Y_t}+\Delta_{Y_t})$-MMP over $X_t$ which terminates. Thus, the exceptional divisor $E_t$ is contracted. As $\Supp E_t = \Supp \Exc(h_t)$, we see that $W_t \to X_t$ is a small $\bQ$-factorial modification.

Finally, by Theorem \ref{thm: deform pic}, after shrinking $T$, for any open subset $U\subset T$, the natural restriction maps
\[
    N^1(W/T) \to N^1(W_U/U) \to N^1(W_t), \quad N^1(X/T) \to N^1(X_U/U) \to N^1(X_t)
\] are isomorphisms for any $t\in U$.
\end{proof}

    The following proposition shows that if $X$ is $\bQ$-factorial and of Fano type over $T$, then any of its birational contraction models can be obtained via a partial MMP over $T$.

\begin{proposition}\label{prop: bc=bcmmp}
If $X$ is a $\mathbb{Q}$-factorial variety of Fano type over $T$, then
\[
\bc(X/T)=\bc(X/T)_{\rm MMP}.
\]
\end{proposition}
\begin{proof}
By definition, we have $\bc(X/T)\supset\bc(X/T)_{\rm MMP}$.

For the inverse inclusion, let $Y/T \in \bc(X/T)$. We need to show that $Y$ can be obtained from $X$ via an MMP over $T$, followed by a birational contraction. More precisely, up to an isomorphism of $Y/T$, we need to show that the birational contraction $X \dto Y/T$ corresponding to $Y$ factors as a $B$-MMP over $T$, $X \dto Z/T$, followed by the birational contraction $Z \to Y/T$ induced by the big and semi-ample divisor $B_Z$.
    
    Let $A'$ be an ample$/T$ divisor on $Y$, and let $A$ be its strict transform on $X$. Let $p: W \to X$ and $q: W \to Y$ be birational morphisms such that $q\circ p^{-1}$ is the map $X \dto Y$. As $X$ is a $\bQ$-factorial variety, we can write
    \[
    p^*A+E=q^*A'+F,
    \] where $E, F$ are effective $q$-exceptional divisors. Let $B=p_*(p^*A+E)$ be a $\bQ$-Cartier divisor on $X$. Then 
    \[
    p^*B+E_1=p^*A+E+E_2,
    \] where $E_1, E_2$ are effective $p$-exceptional divisors. Since $X \dto Y$ is a birational contraction, $E_1, E_2$ are also $q$-exceptional divisors. We have
    \[
    p^*B+E_1=p^*A+E+E_2=q^*A'+F+E_2.
    \] Since $A'$ is ample$/T$ and both $F$ and $E_2$ are $q$-exceptional divisors, the divisors $q^*A'$ and $F + E_2$ form the positive and negative parts, respectively, of the Nakayama-Zariski decomposition of $p^*B + E_1$ over $T$. Let $X \dto Z/T$ be a $B$-MMP$/T$ such that $B_Z$, the strict transform of $B$, is semi-ample$/T$. Replacing $W$ by a higher model, we may assume that there exists a birational morphism $r: W \to Z$. This modification does not change the divisor $B$. We have
    \[
    p^*B = r^*B_Z+E_Z,
    \] where $E_Z$ is an effective $r$-exceptional divisor. Since $E_1$ is also $r$-exceptional, $p^*B+E_1$ admits the Nakayama-Zariski decomposition whose positive part is $r^*B_Z$ and the negative part is $E_Z+E_1$. By the uniqueness of the Nakayama-Zariski decomposition, we have $E_Z+E_1=F+E_2$ and $r^*B_Z=q^*A'$. As $A'$ is ample$/T$, the natural birational map $g: Z \dto Y/T$ is indeed a morphism.
\end{proof}

 \begin{lemma}\label{lem: uniform behavior in bc}
Let $X$ be of Fano type over $T$. Then, up to a dominant generically finite base change of $T$, there exists an open subset $T_0\subset T$ such that the following properties hold:
\begin{enumerate}
\item for any open subset $V\subset T_0$, if $Y/V\in \bc(X_{V}/V)$, then the natural maps
\[
N^1(Y/V) \to N^1(Y_t), \quad \Nef(Y/V) \to \Nef(Y_t), \quad \NE(Y_t) \to \NE(Y/V)
\] are isomorphisms for any $t\in V$;
\item for any $Y/T_0\in \bc(X/T_0)$, $Y$ is flat over $T_0$, and $Y_t$ is an irreducible normal variety for each $t\in T_0$. 

\item if $V\subset T_0$ is an open subset and $g: Y \to Z/V$ is a contraction (not necessarily birational) for some $Y \in \bc(X_{V}/V)$, then $g_t: Y_t \to Z_t$ is still a contraction for each $t\in V$. Moreover, $g$ is a divisorial contraction (resp. a small contraction, a Mori fiber space) if and only if $g_t$ for any $t\in V$ is a divisorial contraction (resp. a small contraction, a Mori fiber space).
\end{enumerate}
 \end{lemma}
 \begin{proof}
By Theorem \ref{thm: fiberwise q-fact}, there exists a dominant generically finite base change $T' \to T$ and an open subset $T_0'\subset T'$ such that $X_{T_0'}$ admits a small $\bQ$-factorial modification $W \to X_{T_0'}$ which is also a fiber-wise small $\bQ$-factorial modification over $T_0'$. Moreover, the natural map $N^1(W/U) \to N^1(W_t), t\in U$ is an isomorphism, where $U\subset T_0'$ is an open subset. Note that $W$ is of Fano type over $T_0'$ after shrinking $T_0'$. By Proposition \ref{prop: bc=bcmmp}, we have
\[
\bc(W/T_0')_{\rm MMP}=\bc(W/T'_{0})= \bc(X_{T_0'}/T_0').
\] Replacing $W$ by $X$ and $T_0'$ by $T$, we may assume that $Y/V \in \bc(X_V/V)_{\rm{MMP}}$ and that the natural map $N^1(X_V/V) \to N^1(X_t)$ is an isomorphism for any $t \in V \subset T$.  By Proposition \ref{prop: strong finiteness} and Proposition \ref{lem: surj preserved for bir contractions} (2), after shrinking $T$, we can assume that for any open subset $V\subset T$ and $Y/V \in \bc(X_V/V)_{\rm MMP}$, $N^1(Y/V) \to N^1(Y_t)$ is surjective for very general $t\in V$. By Theorem \ref{thm: deform pic} (1) and the finiteness of $\bc(X_V/V)_{\rm MMP}$, after shrinking $T$, we can assume that
\[
N^1(Y/V) \to N^1(Y_t), t\in V
\] is an isomorphism for any open subset $V\subset T$ and any $Y/V \in \bc(X_V/V)_{\rm MMP}$. By Corollary \ref{cor: nef cone for MDS}, after shrinking $T$, we can assume that for any $Y/T \in \bc(X/T)_{\rm MMP}$ and any open subset $V\subset T$, the natural maps
 \[
\Nef(Y_V/V) \to \Nef(Y_t),  \quad \NE(Y_t) \to \NE(Y/V)
\] are isomorphisms for each $t\in V$. This shows (1). Since each step in the above process remains valid after shrinking $T$, the statement (1) continues to hold after shrinking $T$.

For (2), after shrinking $T$, we can assume that for any $Y/T\in \bc(X/T)$, $Y$ is flat over $T$ by generic flatness, and $Y_t$ is an irreducible normal variety for each $t\in T$ as $\{p\in T \mid Y_p \text{~is normal}\}$ is a constructible set. Note that, in the above set, $p$ is not restricted to closed points, and the generic point of $T$ belongs to this set as $Y$ is normal.

To show (3), first note that by (1), we have $\Nef(Y/T) \simeq \Nef(Y_V/V)$ for any $Y/T\in \bc(X/T)$ and open subset $V\subset T$. As $\Nef(Y/T)$ is a polyhedral cone, there are only finitely many contractions (not just birational contractions). By Proposition \ref{prop: strong finiteness}, if $h: Y' \to Z'/V$ is a contraction for some $Y'/V \in \bc(X_V/V)$, then there exist $Y/T \in \bc(X/T)$ and a contraction $g: Y \to Z/T$ such that $g_V= h$, where the above equality holds up to an isomorphism of the targets. As $\NE(Y/T) \simeq \NE(Y_V/V) \simeq \NE(Y_t)$, we see that $g, g_V, g_t$ are extremal contractions simultaneously. Shrinking $T$ further, we can assume that $g$ is a divisorial contraction (resp. a small contraction, a Mori fiber space) if and only if $g_t$ is a divisorial contraction (resp. a small contraction, a Mori fiber space) for each $t\in V$.
 \end{proof}

\begin{proof}[Proof of Theorem \ref{mainthm: mmp}]
By Theorem \ref{thm: FT is constructible}, after shrinking $T$, we can assume that $X/T$ is of Fano type. After replacing $T$ by a dominant generically finite base change and possibly shrinking it, we may assume that the properties listed in Lemma \ref{lem: uniform behavior in bc} hold.

First, by Proposition \ref{prop: finiteness} and Proposition \ref{prop: strong finiteness}, it suffices to consider a fixed birational contraction $g: X' \dto Y/U$. Take birational morphisms $p\colon W \to X'$ and $q\colon W \to Y$ such that $g = q \circ p^{-1}$. Shrinking $U$, we may assume that for any $t \in U$, the induced map
\[
g_t \colon X'_t \dashrightarrow Y_t
\]
is still a birational contraction, and that $p_t$ and $q_t$ are birational morphisms. Moreover, we may further assume that if a divisor $E$ on $W$ is $p$-exceptional (resp. $q$-exceptional), then for every $t \in U$, the divisor $E|_{W_t}$ is $p_t$-exceptional (resp. $q_t$-exceptional). 
By construction, for any $\bR$-Cartier divisor $\mD$ on $X'$, we have 
\[
p^*\mD = q^*(g_*\mD) +E,
\] where $E$ is $q$-exceptional. Hence, we have
\[
p_t^*\mD|_{X_t'} \sim_\bR q_t^*((g_*\mD)|_{Y_t}) + E|_{W_t}.
\] As $g_t$ is a birational contraction and $E|_{W_t}$ is $q_t$-exceptional, we see 
\[
g_{t,*}(\mD|_{X_t'}) \sim_\bR (g_*\mD)|_{Y_t}.
\] 

Next, we show that an MMP of the total space is an MMP of each fiber of the same type. For any $\mD \in N^1(X/T)$, let 
\[
X=X_0 \dasharrow X_1 \dasharrow \cdots \dto X_{n-1}\dasharrow X_n \to X_{n+1}
\] be a $\mD$-MMP over $T$,  where $X_i \dasharrow X_{i+1}$ for $i=0, \ldots, n-1$ are birational contractions, and $X_n \to X_{n+1}$ is either a contraction induced by the semi-ample$/T$ divisor $\mD_{X_{n}}$ or a Mori fiber space when $\mD$ is not pseudo-effective$/T$. Thus,  we have $X_i/T \in \bc(X/T)_{\rm MMP}$ for $0 \leq i \leq n$, and $X_{n+1}/T\in \bc(X/T)_{\rm MMP}$ if $X \dasharrow X_{n+1}$ is birational. To be precise, this means that the above natural birational contraction $X \dasharrow X_i/T$ belongs to $\bc(X/T)_{\rm MMP}$. By Lemma \ref{lem: uniform behavior in bc} (3), $g: X_i \dasharrow X_{i+1}$ 
%(resp. $g': X_{i+1} \dasharrow X_{i}$) 
is a divisorial contraction (resp. a small contraction, an extremal contraction, a Mori fiber space) if and only if $g_t$ 
%(resp. $g_t'$) 
is a divisorial contraction (resp. a small contraction, an extremal contraction, a Mori fiber space) for each $t\in T$. Hence, for each $t\in T$,
\[
X_t=X_{0,t} \dasharrow X_{1,t} \dasharrow \cdots \dasharrow X_{n,t} \to X_{n+1, t}
\] is a $\mD_t$-MMP on $X_t$ of the same type. 
%If $X_n \to Z$ is a birational contraction, then, by Lemma \ref{lem: uniform behavior in bc} (1), we know that $\mD_{Z,t}$ is nef (in fact semi-ample as $Z$ is of Fano type over $T$).

Conversely, we show that an MMP of the fiber is induced from an MMP$/T$ of the total space. Suppose that 
\begin{equation}\label{eq: MMP fiber}
X_t=Y_{0} \dasharrow Y_{1} \dasharrow \cdots \dasharrow Y_{n} \to W
\end{equation} 
is a $D$-MMP on $X_t$. By Lemma \ref{lem: uniform behavior in bc} (1), there exists $\mD\in N^1(X/T)$ such that $\mD_t\sim_\bR D$. If $\sigma_0: Y_0 \to Y_1$ is a divisorial contraction (resp. a small contraction, a Mori fiber space), then by Lemma \ref{lem: uniform behavior in bc} (1), there exists a contraction $g_0: X_0 \to X_1$ that contracts the same extremal ray as $\sigma_0$, and thus $g_{0,t}=\sigma_0$. By Lemma \ref{lem: uniform behavior in bc} (3), $g_0$ is still a divisorial contraction (resp. a small contraction, a Mori fiber space). When $\sigma_0$ is a flipping contraction, let $\sigma_1: Y_2 \to Y_1$ be its flip. Let $g_1: X_2 \to X_1$ be the flip of the flipping contraction $g_0: X_0 \to X_1$. By Lemma \ref{lem: uniform behavior in bc} (2) (3), $g_{1,t}: X_{2,t} \to X_{1,t}$ is an extremal contraction between normal varieties. Moreover, $\mD_{X_2}|_{X_{2,t}}$ is ample over $X_{1,t}$ as $\mD_{X_2}$ is ample over $X_1$, where $\mD_{X_2}$ is the strict transform of $\mD$ on $X_2$. Therefore, $g_{1,t}$ is exactly $\sigma_1$. Repeating this process, we obtain a $\mD$-MMP on $X/T$ whose restriction to the fiber $X_t$ is exactly \eqref{eq: MMP fiber}. Moreover, this $\mD$-MMP terminates when \eqref{eq: MMP fiber} terminates by Lemma \ref{lem: uniform behavior in bc} (1) (3).

Finally, if $Y/U\in \bc(X_{U}/U)_{\mathrm{MMP}}$, then (3) directly follows from Lemma \ref{lem: uniform behavior in bc} (1). In general, by Theorem \ref{thm: fiberwise q-fact}, after a dominant generically finite base change of $T$ and shrinking it, there exists a fiber-wise small $\bQ$-factorial modification $W \to X/T$. By Proposition \ref{prop: bc=bcmmp}, we have $Y/U \in \bc(W_U/U)_{\rm MMP}$. Then by the same argument as before, (3) follows from Lemma \ref{lem: uniform behavior in bc} (1).
\end{proof}

\subsection{Deformation of effective cones}\label{subsec: def eff cone}

As an application of Theorem \ref{mainthm: mmp}, we show the following result on the deformation invariance of effective cones.

\begin{proposition}\label{prop: def of eff cone}
Let $f: X \to T$ be a projective fibration. Suppose that there exists a Zariski dense subset $S\subset T$ such that the fibers $X_s$ for $s\in S$ satisfy one of the cases in Theorem \ref{thm: FT is constructible}. Then, up to a dominant generically finite base change of $T$, there exists a non-empty open subset $T_0\subset T$, such that for any $U\subset T_0$, we have
\[
\Eff(X_{T_0}/T_0) \simeq \Eff(X_U/U) \simeq \Eff(X_t), t\in U
\]
 under the natural restriction maps.
\end{proposition}
\begin{proof}
After replacing $T$ by a dominant generically finite base change and possibly shrinking it, we may assume that Theorem \ref{mainthm: mmp} holds. By Theorem \ref{thm: FT is constructible}, after shrinking $T$, we can assume that $X/T$ is of Fano type. If $\mD\in \Eff(X/T)$, then we can run a $\mD$-MMP$/T$, $X \dto X_n$, which terminates at $X_n$ such that $\mD_{X_n}$ is nef $/T$. By Theorem \ref{mainthm: mmp} (2), it is also a $\mD_t$-MMP on $X_t$ for any $t\in T$. This shows that under the identification $N^1(X/T) \simeq N^1(X_t)$ (see Theorem \ref{mainthm: mmp} (4)), we have
\[
\Eff(X/T) \subset \Eff(X_t).
\]

Conversely, take $D\in \Eff(X_t)$. Let $X_t \dasharrow Y$ be a $D$-MMP such that $D_Y$ is nef on $Y$. Then by Theorem \ref{mainthm: mmp} (3), it is the restriction of a $\mD$-MMP$/T$ on $X$, where $\mD$ is a divisor on $X$ such that $[\mD_t]=[D]$. Hence we have $[\mD]\in \Eff(X/T)$. This shows the inverse inclusion $\Eff(X/T) \supset \Eff(X_t)$.
\end{proof}

\subsection{Deformation of movable cones and Mori chamber decompositions}

The deformation invariance of movable cones and Mori chamber decompositions is also a consequence of Theorem \ref{mainthm: mmp}.

If $X$ is of Fano type over $T$ and $[D] \in \Mov(X/T) = \bMov(X/T)$ is a movable divisor, then there exists a $D$-MMP$/T$, $\phi: X \dto Y/T$, which terminates at $Y$. Note that $X \dto Y$ is isomorphic in codimension $1$ and if $D_Y$ is the strict transform of $D$ on $Y$, then $[D_Y]\in \Nef(Y/T)$. 
Conversely, if $\phi: X \dto Y$ is isomorphic in codimension $1$, then $\phi^{-1}_{*}\Nef(Y/T) \cap N^1(X/T) \subset \Mov(X/T)$. Here the chamber $\phi^{-1}_{*}\Nef(Y/T) \cap N^1(X/T)$ consists of $\bR$-Cartier divisors which are strict transforms of nef divisors on $Y/T$. 

Therefore, $\Mov(X/T)$ admits the following Mori chamber decomposition
\begin{equation}\label{eq: Mori chamber decomp1}
\Mov(X/T)=\bigcup_{\substack{\phi: X \dto Y \text{~isom.~in codim.~} 1,\\Y/T \in \bc(X/T)_{\rm MMP}}} \left(\phi^{-1}_{*}\Nef(Y/T) \cap N^1(X/T)\right).
\end{equation} We emphasize that $Y/T \in \bc(X/T)_{\mathrm{MMP}}$ indeed represents the map $\phi$.

This Mori chamber decomposition may differ from 
\begin{equation}\label{eq: Mori chamber decomp2}
\Mov(X/T)=\bigcup_{\phi: X \dto Y \text{~isom.~in codim.~} 1} \left(\phi^{-1}_{*}\Nef(Y/T) \cap N^1(X/T)\right),
\end{equation} where all birational maps that are isomorphic in codimension $1$ are considered. Indeed, if $\psi: \tilde X \to X$ is a non-isomorphic small $\bQ$-factorial modification, then the chamber $\left(\psi_{*}\Nef(\tilde X/T) \cap N^1(X/T)\right)$ appears in \eqref{eq: Mori chamber decomp2}. However, it does not appear in \eqref{eq: Mori chamber decomp1} because $X \dto \tilde X$ does not belong to $\bc(X/T)_{\rm MMP}$. On the other hand, if $X$ is $\bQ$-factorial, then Proposition \ref{prop: bc=bcmmp} implies that the two chamber decompositions are the same. It is the Mori chamber decomposition given in \eqref{eq: Mori chamber decomp1} that will be considered below. (A similar statement can be obtained for the decomposition in \eqref{eq: Mori chamber decomp2} by passing to a $\bQ$-factorial modification, but the formulation is more involved and will be omitted here.)

\begin{proposition}\label{prop: mov}
Let $f: X \to T$ be a projective fibration. Suppose that there exists a subset $S\subset T$ such that the fibers $X_s$ for $s\in S$ satisfy one of the cases in Theorem \ref{thm: FT is constructible}. Then, up to a dominant generically finite base change of $T$, there exists a Zariski open subset $T_0\subset T$, such that for any $U\subset T_0$, we can identify the Mori chamber decompositions of $\Mov(X_{T_0}/T_0), \Mov(X_{U}/U)$ and $ \Mov(X_t), t\in U$ under the natural restriction maps. In particular, we have isomorphisms among movable cones 
\[
\Mov(X_{T_0}/T_0) \to \Mov(X_U/U) \to \Mov(X_t), t\in U
\] under the natural restriction maps.
\end{proposition}
\begin{proof}
After replacing $T$ by a dominant generically finite base change and possibly shrinking it, we may assume that Theorem \ref{mainthm: mmp} holds. By Theorem \ref{thm: FT is constructible}, after shrinking $T$, we can assume that $X/T$ is of Fano type. By Theorem \ref{mainthm: mmp} (2) and (3), an MMP$/T$ on $X$ restricts to an MMP of the same type on each fiber $X_t$, and the converse also holds. Therefore, $\Mov(X/T)$ and $\Mov(X_t)$ share the same chambers under the identification $N^1(X/T) \simeq N^1(X_t)$ (see Theorem \ref{mainthm: mmp} (4)). 
\end{proof}

With these preparations, we are ready to prove Theorem \ref{mainthm: decomposition}.

\begin{proof}[Proof of Theorem \ref{mainthm: decomposition}]
To show (1), by Theorem \ref{thm: FT is constructible}, after shrinking $T$,
we may assume that $X$ is of Fano type over $T$. By Theorem
\ref{thm: deform pic} (1), there exists a non-empty Zariski open subset
$T_0\subset T$ such that 
\begin{equation}\label{eq: iso for NSs}
N^1(X_{T_0}/T_0) \simeq N^1(X_t), t\in T_0.
\end{equation}
In particular, $N^1(X_t)$ is deformation invariant for
$t\in T_0$. After shrinking $T_0$ further, $\Nef(X_t)$ is deformation
invariant for $t\in T_0$ by Corollary \ref{cor: nef cone for MDS}. The invariance of $\Eff(X_t)$, $\Mov(X_t)$, and the Mori chamber
decompositions follows from the proofs of Proposition
\ref{prop: def of eff cone} and Proposition \ref{prop: mov}. Indeed, once
we have \eqref{eq: iso for NSs}, no dominant generically finite base change is needed; it suffices to shrink $T_0$.

For (2), after a dominant generically finite base change of $T$, Theorem
\ref{mainthm: mmp} (4) allows us to assume that
$N^1(X/T) \to N^1(X_t)$ is an isomorphism for every $t\in T$. Hence, (2) follows from (1).
\end{proof}

Let $X/T$ be a relative Mori dream space with $X$ being a $\bQ$-factorial variety. Then the effective cone also admits a chamber decomposition; see \cite{HK00}*{Proposition 1.11 (2)}. More precisely, we have
\begin{equation}\label{eq: chamber of eff}
\Eff(X/T)
=
\bigcup_{g\colon X \dashrightarrow Y/T}
\left(g^*\Nef(Y/T) + \Exc(g)\right),
\end{equation}
where $g$ runs over all birational contractions over $T$, and $\Exc(g)$ denotes
the cone generated by the $g$-exceptional prime divisors. The analogous statement to Proposition \ref{prop: mov} also holds for the effective cone. We are grateful to the referee for bringing this fact to our
attention.

\begin{proposition}\label{prop: Mori chamber for Eff}
Under the same assumptions as in Proposition \ref{prop: mov}, after a
dominant generically finite base change of $T$, there exists a non-empty Zariski open
subset $T_0\subset T$ such that, for any open subset $U\subset T_0$ and any
$t\in U$, the chamber decompositions in \eqref{eq: chamber of eff} of $
\Eff(X_{T_0}/T_0), \Eff(X_U/U)$ and $\Eff(X_t)$ are identified under the natural restriction maps.
\end{proposition}

\begin{proof}
After replacing $T$ by a dominant generically finite base change and possibly shrinking
it, we may assume that Theorem \ref{mainthm: mmp} holds. By Theorem
\ref{thm: FT is constructible}, after shrinking $T$ further, we may assume that
$X/T$ is of Fano type. By Theorem \ref{mainthm: mmp} (2) and (3), an MMP$/T$ on $X$ restricts to an
MMP of the same type on each fiber $X_t$, and conversely every MMP on $X_t$ is
induced by an MMP$/T$. By Theorem
\ref{mainthm: mmp} (4), for any $Y/U\in \bc(X_U/U)$ and any $t\in U$, the
natural maps
\[
N^1(Y/U) \to N^1(Y_t), \quad
\Nef(Y/U) \to \Nef(Y_t)
\]
are isomorphisms.

It remains to compare the exceptional parts of the chambers. Shrinking $T_0$
further, we may assume that, for every birational contraction
$g\colon X_{T_0}\dashrightarrow Y/T_0$, every $g$-exceptional prime divisor
dominates $T_0$, and its restriction to each fiber is a $g_t$-exceptional
prime divisor. Conversely, by the compatibility of the MMP with fibers
described above, every $g_t$-exceptional divisor arises in this way. Hence the
cones generated by the exceptional divisors are identified under the natural
restriction maps.

Therefore, for any open subset $U\subset T_0$ and any $t\in U$, the
decompositions
\[
\Eff(X_U/U)
=
\bigcup_{g\colon X_U\dashrightarrow Y/U}
\left(g^*\Nef(Y/U)+\Exc(g)\right)
\]
and
\[
\Eff(X_t)
=
\bigcup_{g_t\colon X_t\dashrightarrow Y_t}
\left(g_t^*\Nef(Y_t)+\Exc(g_t)\right)
\]
have the same chambers under the above identifications.
\end{proof}

\section{Boundedness of birational models of Fano type varieties}\label{sec: bdd FT}

Based on the results developed in the previous sections, we present an application to the moduli problem of Fano type varieties. 

\begin{proof}[Proof of Theorem \ref{mainthm: bdd}]
By Noetherian induction and Theorem \ref{thm: fiberwise q-fact}, after replacing $T$ by a dominant generically finite base change and possibly shrinking it, we can assume that there exists a fiber-wise small $\bQ$-factorial modification $Y \to X$. Note that  ${\rm bcm}(Y_t) = {\rm bcm}(X_t)$. Replacing $Y$ by $X$, we can assume that $X_t, t\in T$ are $\bQ$-factorial. 

After replacing $T$ by a dominant generically finite base change and possibly shrinking it further, we may assume that Theorem \ref{mainthm: mmp} holds. Note that this replacement does not affect the $\bQ$-factoriality of the fibers. By Theorem \ref{thm: FT is constructible} and Noetherian induction, after shrinking $T$, we can assume that $X$ is of Fano type over $T$. 

Fix a fiber $X_t$ for some $t\in T$, and suppose that $X_t \dasharrow Z$ is a birational contraction. As $X_t$ is a $\bQ$-factorial variety, we have $Z\in \bc(X_t)_{\rm MMP}$ by Proposition \ref{prop: bc=bcmmp}. By Theorem \ref{mainthm: mmp} (3), there exists an MMP$/T$, denoted by $X \dasharrow \mZ/T$, such that $Z \simeq \mZ_t$. Note that $\mZ/T\in \bc(X/T)$. By Proposition \ref{prop: finiteness}, $\bc(X/T)$ is a finite set. Thus, ${\rm bcm}(X/T)$ is bounded by definition.
\end{proof}

\begin{remark}\label{rem: example}
In contrast to Theorem \ref{mainthm: bdd}, boundedness fails even when considering crepant models of a fixed log pair.

For example, let $X = \bP^2$, and let $D$ and $B$ be two simple normal crossing curves on $X$ with $p \in D \cap B$. Set $c_n = 1 - \frac{1}{n}$ for $n \in \bZ_{>1}$. Suppose that $\pi_1: X_1= {\rm Bl}_p X \to X$ is the blow-up of $X$ at $p$ and let $D_1, B_1$ be the strict transforms of $D, B$ on $X_1$, respectively. Let $E_1$ be the exceptional divisor of $\pi_1$. Then we have
  \[
  K_{X_1}+c_nD_1+c_nB_1+(1-\frac 2 n)E_1=\pi_1^*(K_X+c_nD+c_nB).
  \] Next, blowing up $D_1 \cap E_1$, we get a variety $X_2$ with $D_2, B_2$ the strict transforms of $D_1, B_1$ on $X_2$, respectively. We still use $E_1$ to denote the strict transform of $E_1$ on $X_2$, and $E_2$ to denote the exceptional divisor of $X_2 \to X_1$. Then we have
  \[
  K_{X_2}+c_nD_2+c_nB_2+(1-\frac 2 n)E_1+(1-\frac 3 n)E_2=\pi_2^*(K_X+c_nD+c_nB),
  \] where $\pi_2: X_2 \to X$ is the corresponding morphism. We continue this process by successively blowing up each intersection $D_i \cap E_i$ at every step. This yields the equation
  \[
    K_{X_m}+c_nD_m+c_nB_m+(1-\frac 2 n)E_1+\cdots +(1-\frac{m+1}{n})E_m=\pi_m^*(K_X+c_nD+c_nB),
  \] where the notation has the same meaning as above. In particular, there exists a divisor $E_{n-1}$ whose discrepancy with respect to $(X, c_nD+c_nB)$ is $0$. By \cite{BCHM10}*{Corollary 1.4.3}, there exists a birational morphism $\theta_n\colon Y_n \to X$ whose only exceptional divisor is $E_{n-1}$. Therefore, we have
  \[
  K_{Y_n}+c_nD_{Y_n}+c_nB_{Y_n}=\theta_n^*(K_X+c_nD+c_nB) 
  \] with $E_{n-1}$ the only exceptional divisor of $\theta_n$, where $D_{Y_n}, B_{Y_n}$ are strict transforms of $D, B$ on $Y_n$, respectively. In other words, $(Y_n, c_nD_{Y_n}+c_nB_{Y_n})$ is a crepant model of $(X, c_nD+c_nB)$. Note that $\{E_n \mid n \in \bZ_{\geq 1}\}$ consists of distinct divisors.

  We claim that $\{Y_n\mid n \in \bZ_{\geq 2}\}$ does not belong to a bounded family. If $\{Y_n\mid n \in \bZ_{\geq 2}\}$ belongs to a bounded family $\mY \to T$, then without loss of generality, we can assume that there exists a Zariski dense subset $S \subset T$ parametrizing a subset of $\{Y_n\mid n \in \bZ_{\geq 2}\}$. After shrinking $T$, we can assume that it satisfies the assumptions in Theorem \ref{mainthm: mmp}. Possibly shrinking $T$ further, by Theorem \ref{mainthm: decomposition}, we can assume that there exists an effective divisor $\mathcal E\subset\mY$ such that $\mathcal E_{s_0}$ is $\bQ$-linearly equivalent to the exceptional divisor of $\mY_{s_0} \simeq Y_i \to X=\bP^2$ for a fixed $s_0\in S\cap T$. By Theorem \ref{mainthm: mmp}, we can contract $\mathcal E$, and obtain the morphism $\Theta: \mY \to \mX/T$, which is exactly $\mY_{s_0} \simeq Y_i \to X=\bP^2$ over $s_0$. By the rigidity of $\bP^2$, we have $\mX_{T'} \simeq \bP^2 \times T'$, where $T' \to T$ is a dominant finite base change (\cite{Har10}*{Example 5.3.1, Exercise 24.7(c)}). Replacing $T'$ by $T$, we can assume that $\mX \simeq \bP^2 \times T$. Note that for any $Y_j$, there exists a unique morphism $\theta_j: Y_j \to \bP^2$, which is the one obtained in the above construction. Indeed, the exceptional curve of $\theta_j$ must be an exceptional curve of any other $Y_j \to \bP^2$ by the negativity of the intersection number. From this, we know that if $s_j\in S \cap T$ with $\mY_{s_j} \simeq Y_j$, then $\Theta|_{\mY_{s_j}} = \theta_j$. Now, by the construction of the divisors $E_n$, there exists a rational function $f\in K(X)$ such that $\nu_{E_i}(f)=0$ but $\nu_{E_j}(f)>0$ for each $j>i$, where $\nu_{E_n}$ denotes the valuation corresponding to $E_n$. This is because for any $j>i$, $E_j$ is obtained by further blow-ups over a point of $E_i$. Now we take $f$ as a rational function on $\mX = \bP^2 \times T$. By the above discussion, the zero set $\mathcal Z$ of $f$ on $\mY$ contains $E_j = \mathcal E|_{s_j}$, where $j>i$ (as $\nu_{E_j}(f)>0$). By the density of $S$, we have $\mathcal Z \supset \mathcal E$. However, this contradicts the fact that $\nu_{E_i}(f) =0$ (i.e., $f$ does not vanish on $E_i = \mathcal E|_s$). This shows that the set $\{Y_n\mid n \in \bZ_{\geq 2}\}$ is not bounded.
\end{remark}

\bibliography{reference.bib}
\end{document}